\documentclass[10pt,reqno]{amsart}
\usepackage{amsmath,amsthm,amssymb,mathrsfs,stmaryrd,color}
\usepackage[all]{xy}
\usepackage{url}
\usepackage[utf8]{inputenc}
\usepackage[T1]{fontenc}
\usepackage[UKenglish]{babel}
\usepackage[margin=20mm]{geometry}
\usepackage{bbm,bm}
\usepackage{extarrows} 

\usepackage{graphicx}
\usepackage{amsmath,amssymb,amsfonts,amsthm}
\usepackage{mathrsfs,eucal,dsfont}
\usepackage{verbatim,enumitem}
\usepackage{enumitem}
\usepackage{amssymb, amsmath, amsthm, amscd}
\usepackage{eucal,mathrsfs,dsfont}
\usepackage[colorlinks=true,hyperindex, linkcolor=magenta, pagebackref=false, citecolor=cyan,pdfpagelabels]{hyperref}

\newtheorem{theorem}{Theorem}[section]
\newtheorem{lemma}[theorem]{Lemma}

\newtheorem{proposition}[theorem]{Proposition}
\newtheorem{corollary}[theorem]{Corollary}
\numberwithin{equation}{section}

\theoremstyle{definition}

\newtheorem{remark}[theorem]{Remark}

\newcommand\re{{R_e^{0,t}}}
\newcommand\hre{{\hat{R}_e^{0,t}}}
\newcommand{\rmA}{{\mathrm{A}}}

\newcommand{\aAn}{{ \mathrm{A}_N(e)}}

\newcommand{\sO}{\mathbb{G}}
\newcommand\G{{\mathbb G}}

\newcommand\sC{{\mathscr C}}

\newcommand{\rmD}{{\mathrm{D}}}
\newcommand{\rmtlD}{{\tilde{\mathrm{D}}}}

\newcommand{\rmtlTmt}{{\tilde{\mathrm{T}}^t_{M}}}

\newcommand{\rmtlTms}{{\tilde{\mathrm{T}}^s_{M}}}

\newcommand{\rmTmt}{{\mathrm{T}^t_{M}}}

\newcommand{\tlpi}{{\tilde{\pi}}}
\newcommand{\tlpit}{{\tilde{\pi}^{t}}}

\newcommand{\tlpim}{{\tilde{\pi}_M}}

\newcommand{\tlpimt}{{\tilde{\pi}_{M}^{t}}}

\newcommand{\rmT}{{\mathrm{T}}}

\newcommand{\rmtlT}{{\tilde{\mathrm{T}}}}
\newcommand{\rmtlTm}{{\tilde{\mathrm{T}}_M}}

\newcommand\cE{{\kE(\Z^d)}}

\newcommand\pp{{\mathbb P}}

\newcommand{\pr}{\mathbb{P}}
\newcommand\cA{{\mathcal A}}
\newcommand\cB{{\mathcal B}}

\newcommand\kC{{\mathcal C}}

\newcommand\cF{{\mathcal F}}
\newcommand\cG{{\mathcal G}}

\newcommand\kP{{\mathcal P}}
\newcommand\kE{{\mathcal E}}

\newcommand\cU{{\mathcal U}}
\newcommand\cV{{\mathcal V}}

\newcommand\cO{{\mathcal O}}

\newcommand\cQ{{\mathcal Q}}

\newcommand\bH{{\mathbb H}}
\newcommand\N{{\mathbb N}}

\newcommand\Z{{\mathbb Z}}
\newcommand\R{{\mathbb R}}

\newcommand\E{{\mathbb E}}
\newcommand\I{{\mathbb I}}
\newcommand\bO{{\mathbb O}}

\newcommand{\var}{\operatorname{Var}}
\newcommand{\corr}{\operatorname{Corr}}

\newcommand{\influ}{\operatorname{Inf}}

\newcommand{\Inf}{\textrm{Inf}}

\DeclareMathOperator{\Diam}{Diam}

\DeclareMathOperator{\Var}{Var}
\DeclareMathOperator{\Cov}{Cov}

\def\ben#1{\begin{equation}#1\end{equation}}

\begin{document}
\allowdisplaybreaks
\title[Superconcentration and chaos in Bernoulli  percolation]{Superconcentration and chaos in Bernoulli percolation}

\author{Van Quyet Nguyen}
\thanks{\emph{V.Q. Nguyen:}
Institute of Mathematics, Vietnam Academy of Science and Technology, 18 Hoang Quoc Viet, Cau Giay, Hanoi, Vietnam, nvquyet@math.ac.vn}             
\date{}
\maketitle
\begin{abstract}
We study the chemical distance of supercritical Bernoulli percolation on $\Z^d$. Recently, Dembin \cite{dembin2022variance} showed that the chemical distance exhibits sublinear variance, a phenomenon now referred to as superconcentration. In this article, we establish an equivalence between this phenomenon and chaotic behavior of geodesics under small perturbations of the configuration, thereby confirming Chatterjee’s general principle relating anomalous fluctuations to chaos in the context of Bernoulli percolation. Our methods rely on a dynamical version of the effective radius, refining the notion first proposed in \cite{can2024subdiffusive}, in order to measure the co-influence of a given edge whose weight may be infinite. Together with techniques from the theory of lattice animals, this approach allows us to quantify the total co-influence of edges in terms of the overlap between original and perturbed geodesics.
\end{abstract}
\allowdisplaybreaks
\section{Introduction}
\subsection{Model and main result} 
\textit{Bernoulli percolation} is a well-known model for describing the propagation of fluid into a porous medium. This model was introduced by Broadbent and Hammersley in 1957 \cite{broadbent1957percolation} and has been one of the most studied probabilistic models. Let $d \ge 2$, and let $(\mathbb{Z}^d, \cE)$ be the graph where $\mathbb{Z}^d$ is the set of integer lattice points and $\cE$ is the set of nearest-neighbor pairs in $\mathbb{Z}^d$ under the Euclidean norm. Given the parameter $p \in [0,1]$, each edge $e \in \cE$ is chosen randomly to be either retained (\emph{open}) with probability $p$ and deleted (\emph{closed}) otherwise, independently of the state of other edges. Percolation theory is primarily concerned with the geometry of the connected components (\emph{clusters}) of the \emph{percolation graph} $$\cG_p:=(\Z^d, \{e \in \cE: e \text{ is open}\}),$$  
the subgraph formed by the open edges. The most important feature of this model is a simple phase transition: there exists a non-trivial \emph
{critical density} $p_c(d) \in (0,1)$ such that  there is  almost surely a unique infinite open cluster $\kC_\infty$ if $p>p_c(d)$, whereas all open clusters are finite if $p<p_c(d)$. We refer to \cite{grimmett1999percolation,duminil2018sixty} for a survey on the subject. In this paper, we focus on the \textit{chemical distance}, the graph distance in the percolation graph when $p > p_c(d)$.

 For any sets $A,B,U \subset \Z^d$, we define
 \begin{align*}
      \rmD_{U} (A,B) 
      & := \inf \{ |\gamma|: x \in A,\, y\in B,\,\gamma \text{ is an open nearest-neighbor path from } x \text{ to } y \text{ inside } U\}.
 \end{align*}
 When $U = \Z^d$, we simply write $\rmD$ for $\rmD_{\Z^d}$. We observe that the chemical distance between two points is infinite as soon as these points are not in the same open cluster. To overcome this issue, we consider a regularized version as follows:  for each $z \in \Z^d$, we denote by $[z]$ the closest point to $z$ in $\kC_\infty$ for the $\| \cdot \|_{\infty}$-distance, with a deterministic rule breaking ties. This point is known as the regularized point of $z$. Then, we define  
\begin{align*}
 \forall x,y \in \Z^d, \quad \rmtlD(x,y):= \rmD([x],[y]) = \inf \{ |\gamma|: \gamma \text{ is an open nearest-neighbor path from } [x] \text{ to } [y]  \}. 
\end{align*}
The first-order growth of $\rmtlD$ was described by Garet and Marchand \cite{garet2004asymptotic}: for any $p > p_c(d)$, there exists a function $\mu_p: \R^d \to [0,\infty) $ such that,  
\begin{align}\label{time constnat of D*}
  \forall x \in \R^d, \quad  \lim_{n \to \infty} \dfrac{\rmtlD(0,\lfloor nx \rfloor)}{n} = \mu_p(x) \quad \text {a.e and in } L^1,
\end{align}
where $\lfloor \cdot \rfloor$ stands for the floor function. The value $\mu_p(x)$ is called  the \textit{time constant}. Moreover, this time constant is Lipschitz-continuous with respect to $p \in [p_0,1]$ for all $p_0 > p_c(d)$ \cite{cerf2022time,can2023lipschitz}. Naturally, the next question of interest concerns the fluctuation and deviation of the chemical distance. The large deviation of  $\rmtlD$ has been well studied since 1990s, see, for example, \cite{antal1996chemical,garet2007large,dembin2022upper}. As for fluctuations, Dembin established the best current upper bound on the variance of chemical distance (\cite[Theorem 1.1]{dembin2022variance}): there exists a positive constant $C$ such that 
\ben{ \label{subb}
\forall n  \geq 1,\, \forall x \in \Z^d, \quad \Var(\rmtlD(0,nx)) \leq C \frac{n}{ \log n}.
}
This phenomenon was coined \textit{superconcentration} by Chatterjee \cite{chatterjee2014superconcentration}. Roughly speaking, superconcentration refers to situations where a random object exhibits significantly stronger concentration than what is predicted by classical concentration techniques. In fact, classical inequalities, such as the Efron–Stein or Poincaré estimates, typically yield only linear (and thus suboptimal) upper bounds on the order of the variance. Some typical examples of the superconcentration phenomenon can be found in \cite{benjamini2003improved, damron2015sublinear,chatterjee2023superconcentration,dembin2024superconcentration}. In the physics literature, superconcentration, typically investigated via fluctuation exponents, and the existence of \textit{chaos} have been studied for a long time, largely at a heuristic level. For example, the seminal works of Fisher and Huse \cite{fisher1991directed}, Mézard \cite{mezard1990glassy}, Zhang \cite{zhang1987ground} showed that "superconcentration implies chaos" in the context of directed polymers in random environments. Other influential and widely cited contributions to this line of research in spin glass models include the works of Bray and Moore \cite{bray1987chaotic} and  McKay, Berger, and Kirkpatrick \cite{mckay1982spin}. 

The first rigorous mathematical theory of chaos was developed by  Chatterjee \cite{chatterjee2008chaos,chatterjee2009disorder,chatterjee2014superconcentration}, who showed that superconcentration is intimately connected with chaos in certain Gaussian disordered systems. In this framework, the system is said to be chaotic if its ground state is drastically altered by a small perturbation of the energy landscape. Inspired by this work, Ganguly and Hammond \cite{ganguly2024stability} established a sharp phase transition that pinpoints precisely the critical point at which the transition
from stability to chaos occurs in the context of Brownian last passage percolation-a typical example in the class of solvable models. However, the analytic methods in these papers rely essentially on the Gaussian setting. Later, beyond Gaussian environments, this connection was established for the top eigenvector of Wigner matrices by Bordenave, Lugosi, Zhivotovskiy \cite{bordenave2020noise} and in first passage percolation satisfying a finite second moment assumption of edge-weights by Ahlberg, Deijfen, Sfragara \cite{ahlberg2023chaos}. While the popular approach to proving superconcentration relies on hypercontractive inequalities, establishing chaos opens an unexpected path toward significantly improving upper bounds on the variance for non-solvable models; see, for instance,~\cite{basu2023}. The main purpose of the present paper is to demonstrate an analogous connection for the chemical distance in Bernoulli percolation. 
%Roughly speaking, we show that the superconcentration of the chemical distance is equivalent to chaotic phenomenon of its geodesic.
Consequently, together with \eqref{subb}, we claim that the interesting chaotic phenomenon holds true in Bernoulli percolation. 

 We consider Bernoulli percolation as a first passage percolation with the i.i.d. weights $(\omega_e)_{e \in \cE}$ defined as
\begin{align*}
\omega_e =     \begin{cases}
       1, & \text{ with probability } p;\\
        \infty, & \text{ with probability } 1 -p.    \end{cases}
\end{align*}
We emphasize that the finite second moment assumption is no longer true in this first passage percolation. Let $\omega'_e$ be an independent copy of $\omega_e$ for each $e \in \cE$. Let $(U_e)_{e \in \cE}$ be an independent family of independent random variables
uniformly distributed on the interval $[0,1]$. For $t \in [0,1]$, we define a collection of \textit{dynamical} weights $(\omega_e(t))_{e\in \cE}$ as
\begin{align*}
    \omega_e(t) := \begin{cases}
    \omega_e       & \quad \text{if }  U_e \geq t; \\
    \omega'_e  & \quad \text{if } U_e < t,
  \end{cases}
\end{align*}
i.e. $\omega_e(t)$ is obtained by resampling $\omega_e$ with probability $t$. Note that $\omega_e(t)$ has the same distribution as $\omega_e$, with $\pr(\omega_e(t) = 1)= p$ for all $t\geq 0$. Let $\kC_\infty^{t}$ denote the unique infinite cluster in the $t$-noise environment. Let us define the chemical distance with respect to the $t$-noise environment  $ (\omega_e(t))_{e \in \cE}$  as 
$$\rmtlD^t(x,y): =\rmD^t([x]_{t},[y]_{t}),$$
where for all $z \in \Z^d$, the regularized point $[z]_{t}$ denotes the closest point to $z$ in $\kC^{t}_\infty$ (in $\|\cdot\|_\infty$-distance) with a deterministic rule breaking ties. Let $\tlpi^t(0,nx)$ be the \textit{geodesic} of $\rmtlD^t(0,nx)$. If there are several choices for the geodesic, we define $\tlpi^t(0,nx)$ as the intersection of all such geodesics. For convenience, we simply write $\tlpi(0,nx)$ instead of $\tlpi^0(0,nx)$. Let $\vec{p}_c(d)$ be the critical point in oriented bond percolation on $\Z^d$ with remark that $ p_c(d)<  \vec{p}_c(d)$. From now on, we assume that
\begin{align}\label{Useful assump}
\tag{BK} p_{c}(d) < p <     \vec{p}_c(d),
\end{align}
which is known as the Berg–Kesten condition.
Under this assumption, by Proposition \ref{Thm: intersection of geo}, there exists $c \in (0,1)$ such that 
$$cn \leq \E[|\tlpi(0,nx)|]\leq c^{-1}n . $$
We say that the geodesic of $\rmtlD(0,nx)$ is chaotic
    if there exists a sequence $t_n \to 0$ (as $n \to \infty)$ such that 
    \begin{align*}
        \E[| \tlpi(0,nx) \cap \tlpi^{t_n}(0,nx)|] = o(n).
    \end{align*}
Roughly speaking, this means that resampling a small $t$-fraction of the edge states results in only a vanishing fraction of the edges belonging to the original geodesic remaining in the geodesic of the resampled configuration.
   
Our first result establishes a connection between the variance of the chemical distance and the dynamical behavior of the overlap between two geodesics.
\begin{theorem} \label{Theorem: variance and chaos} There exists a positive constant $ C $ such that for all $n \geq  1$ and $x \in \Z^d$,
\begin{align}\label{Thm: lower}
     \frac{1}{C} \int_0^1 \E[|\tlpi(0,nx) \cap \tlpi^{t}(0,nx)|] dt - C(\log n)^{d+8}\leq \var(\rmtlD(0,nx)),
\end{align}
and for all $\alpha> 0$,
\begin{align}\label{Thm: upper}
      \var(\rmtlD(0,nx))\leq C \alpha^2 \int_0^1 \E[|\tlpi(0,nx) \cap \tlpi^{t}(0,nx)|] dt + Ce^{-\alpha/C} n + C (\log n)^{2d+6}.
\end{align}
Consequently, as $n \to \infty$, we obtain
\begin{align*}
    \var(\rmtlD(0,nx)) = o(n) \Longleftrightarrow \exists t_n \to 0, \, \E[|\tlpi(0,nx) \cap \tlpi^{t_n}(0,nx)|] = o(n).
\end{align*}
\end{theorem}
For any $t \geq 0$, let us define 
\begin{align*}
    \corr(\rmtlD(0,nx),\rmtlD^{t}(0,nx))
    := \frac{\Cov(\rmtlD(0,nx),\rmtlD^{t}(0,nx))}{\var(\rmtlD(0,nx))}.
\end{align*}
The next result characterizes the stable and chaotic regimes of the chemical distance. 

\begin{theorem}\label{theorem: main} Let $\hat{t}_n := \frac{\Var(\rmtlD(0,nx))}{n}$. Then there exists a positive constant $C$ such that for all $ 0 < \beta < 1/\hat{t}_n$ and $x \in \Z^d$,
\begin{itemize}
    \item [(i)] (Stable regime) if $ \frac{t_n} {\hat{t}_n}\leq \beta  $ then 
    $$\corr(\rmtlD(0,nx),\rmtlD^{t_n}(0,nx)) \geq 1 - C \beta .$$
\item [(ii)] (Chaotic regime) if $ \frac{t_n} {\hat{t}_n} \geq \beta$ then
    $$\E[|\tlpi(0,nx) \cap \tlpi^{t_n}(0,nx)|] \leq C\left(\frac{n}{\beta} + (\log n)^{d+8}\right).
    $$
Consequently, if $t_n \to 0$ and $t_n \log n \to \infty $ as $n \to \infty$ then
\begin{align*}
    \E[|\tlpi(0,nx) \cap \tlpi^{t_n}(0,nx)|] = o(n).
\end{align*}
\end{itemize}
\end{theorem}
The above phase transition exhibits a sharp window around the critical sequence $\hat{t}_n$. In the chaotic regime, if $t_n \gg \hat{t}_n$, we have
    $$\E[|\tlpi(0,nx) \cap \tlpi^{t_n}(0,nx)|] = o(n).
$$
 In the stable regime, whereas if $t_n \ll \hat{t}_n$ then
    $$\corr(\rmtlD(0,nx),\rmtlD^{t_n}(0,nx)) \geq 1 -o(1).$$
In this regime, we say that the chemical distance $\rmtlD(0,nx)$ is \textit{noise stable}. By analogy with first passage percolation, we expect that in the chaotic regime the chemical distance $\rmtlD(0,nx)$ is \textit{noise sensitive} (an opposite notion),
i.e for $t_n \gg \hat{t}_n$,
\begin{align}\label{Eq: noise sensitivity}
    \corr(\rmtlD(0,nx),\rmtlD^{t_n}(0,nx)) = o(1).
\end{align}
Noise sensitivity was introduced in the context of Boolean functions in the pioneering work of Benjamini, Kalai, and Schramm \cite{benjamini1999noise}, which captures the phenomenon whereby the output becomes asymptotically independent under two highly correlated inputs. From the perspective of spectral theory, noise sensitivity and superconcentration correspond to the concentration of Fourier mass on the high end of the spectrum. This connection therefore provides evidence for an equivalence between superconcentration and chaos.  Whenever \eqref{Eq: noise sensitivity} holds, a stronger transition is obtained as in \cite{ganguly2024stability,bordenave2020noise}, in the sense that both the stable and
chaotic regimes involve the same object (chemical distance). However, such noise sensitivity is not currently known for growth models such as first passage percolation or supercritical Bernoulli percolation, where the ground states take real values rather than Boolean ones.
\begin{remark}
In the proof of Theorem \ref{Theorem: variance and chaos}, the assumption \eqref{Useful assump} is used only to obtain a uniform lower bound on the dynamical overlap (see \eqref{Eq: lower bound of intersect pi0,s}).
This lower bound allows us to eliminate certain decay error terms in the upper bound of the variance. If we can establish a sufficiently sharp lower bound on the dynamical overlap, then the relations \eqref{Thm: lower} and \eqref{Thm: upper} remain valid without the poly-logarithmic error terms.
\end{remark}
\subsection{Sketch of the proofs} Fix $M := M(n)= \lfloor (\log n)^2 \rfloor$. We consider the first truncated passage percolation model in which  the family of the edge-weights $\tau:=(\tau_e)_{e \in \cE}$ is an $M$-truncated version of the original weight $\omega$. In other words, $(\tau_e)_{e \in \cE}$ are i.i.d. random variables distributed as
\begin{align*}
    F_M:= p\delta_1+ (1-p) \delta_M,
\end{align*}
where $\delta_a$ is the Dirac measure with the mass at $a$. For any set $A \subset \Z^d$, we denote the first truncated passage time by
\begin{align}\label{}
    \rmT_{A,M}(x,y) := \inf \left\{ \sum_{e \in \gamma} \tau_e: \gamma \text{ is a self-avoiding path from } x \text{ to } y \text{ inside } A \right\},
\end{align}
with the convention that when $A = \Z^d$, we simply write $\rmT_M$.
We also define 
\begin{align*}
    \rmtlT_M(x,y):= \rmT_{M}([x],[y]).
\end{align*}
For $t \in [0,1]$, we define a collection of $t$-noise environments $\tau(t)=(\tau_e(t))_{e\in \cE}$ as the $M$-truncated version of the edge-weights $\omega(t)$, i.e. for each $e \in \cE$,
\begin{align*}
    \tau_e(t)= \I(U_e \geq t) \tau_e + \I(U_e < t) \tau'_e,
\end{align*}
where $\tau':=(\tau'_e)_{e\in \cE}$  is an $M$-truncated version of $\omega'$. We also denote $ \rmT^t_{M}$ the first truncated passage time in $t$-noise environment $\tau(t)$ and 
\begin{align*}
    \rmtlT^t_M(x,y):= \rmT_{M}^t([x]_{t},[y]_{t}),
\end{align*}
where we recall that for all $z \in \Z^d$, $[z]_{t}$ denotes the closest point to $z$ in $\kC^{t}_\infty$. Let $\tlpi^t_M(0,nx)$ be the geodesic of $\rmtlT^t_M(0,nx)$, with the convention that $\tlpi_M(0,nx):= \tlpi^0_M(0,nx)$. If the geodesic is not unique, we define $\tlpi^t_M(0,nx)$ to be the intersection of all such geodesics. We also let $\G^t_M([x]_t,[y]_t)$ denote the set of all geodesics of $ \rmtlT^t_M(x,y)$. 
 
Our idea of the proof is as follows. First, we approximate the chemical distance $\rmtlD$ by the first truncated passage time $\rmtlT_M$ uniformly in $t$. Next, we derive dynamical variance inequalities for $\rmtlT_M$ via a representation of the variance in terms of the dynamical \textit{co-influences} of respective edges. This representation allows us to reduce the problem to estimating the co-influence of a single edge with respect to $\rmtlT_M$ at time $t$. The main idea for obtaining upper bounds on both the co-influence and the discrepancy between chemical distance and first truncated passage time is to control the effect of resampling an edge-weight using the dynamical version of \textit{effective radius}. Finally, we manage the total resampling cost over all edges in $\cE$ by applying the theory of lattice animals in dependent environments. We now describe the outline in more detail.
\\
\noindent \underline{\textbf{Step 1}}: \textbf{Approximating chemical distance by the first truncated passage time}.
We will control the difference in the overlap between geodesics as follows:
    \begin{align}\label{Eq: gedesics size difference}
      \forall t\geq 0, \quad  \E[| |\tlpi(0,nx) \cap \tlpit(0,nx)|- |\tlpim(0,nx) \cap \tlpimt(0,nx)||] = \cO(\exp(-M^{3/4})). 
    \end{align} 
Next, we aim to show that
    \begin{align}\label{Eq: variance difference}
     \forall t\geq 0, \quad   |\Var(\rmtlD^t(0,nx)) - \var(\rmtlTmt(0,nx))| = \cO(\exp(-M^{3/4})).
    \end{align}
Let $\omega^1:=(\omega^1_e)_{e \in \cE}$ and $\omega^2:=(\omega^2_e)_{e \in \cE}$ be two independent copies of $\omega,\omega'$. Let $\tau^1:=(\tau^1_e)_{e \in \cE}$ and $\tau^2:=(\tau^2_e)_{e \in \cE}$ be the $M$-truncated versions of $\omega^1$ and $\omega^2$, respectively. Then, $\tau^1$ and $\tau^2$ are two independent copies of $\tau,\tau'$. Let us define the discrete derivatives of an edge $e$ with respect to $\rmtlD(0,nx)$ and $\rmtlT_M(0,nx)$ as 
\begin{align*}
    & \nabla^{\omega_e,\omega^1_e}_e \rmtlD(0,nx):= \rmtlD(0,nx) \circ \sigma_e^{\omega_e}(\omega)- \rmtlD(0,nx) \circ \sigma_e^{\omega^1_e}(\omega),\\
    & \nabla^{\omega_e,\omega^2_e}_e \rmtlD^t(0,nx):= \rmtlD^t(0,nx) \circ \sigma_e^{\omega_e}(\omega(t))- \rmtlD^t(0,nx) \circ \sigma_e^{\omega^2_e}(\omega(t)); \\
    & \nabla^{\tau_e,\tau^1_e}_e \rmtlT_M(0,nx):= \rmtlT_M(0,nx) \circ \sigma_e^{\tau_e}(\tau)- \rmtlT_M(0,nx) \circ \sigma_e^{\tau^1_e}(\tau),\\
    & \nabla^{\tau_e,\tau^2_e}_e \rmtlT^t_M(0,nx):= \rmtlT^t_M(0,nx) \circ \sigma_e^{\tau_e}(\tau(t))- \rmtlT^t_M(0,nx) \circ \sigma_e^{\tau^2_e}(\tau(t)),
\end{align*}
where for any vector $X:=(X_e)_{e \in \cE}$, the operator $\sigma_e^a(X)$ denotes the configuration obtained from $X$ by replacing the coordinate $X_e$ with $a$.
We then define the corresponding co-derivatives by
\begin{align*}
    \Delta_e(\rmtlD,\rmtlD^t):= \nabla^{\omega_e,\omega^1_e}_e \rmtlD(0,nx) \nabla^{
     \omega_e,\omega^2_e}_e \rmtlD^t(0,nx),\quad  \Delta_e(\rmtlTm,\rmtlT^t_M):=\nabla^{\tau_e,\tau^1_e}_e \rmtlT_M(0,nx) \nabla^{\tau_e,\tau^2_e}_e \rmtlT^t_M(0,nx).
\end{align*}
Taking expectations of these quantities yields the co-influences of the edge $e$ with respect to the chemical distance and the first truncated passage time at time $t$, respectively. For all $ t \geq 0$, we aim to estimate the difference between the corresponding total co-influences of edges:
        \begin{align}\label{Eq: comparison of influence}
     \left|\sum_{e \in \cE}\E\left[\Delta_e(\rmtlD,\rmtlD^t)\right] -\sum_{e \in \cE}\E\left[ \Delta_e(\rmtlTm,\rmtlT^t_M) \right]\right|  = \cO(\exp(-M^{3/4})).
\end{align}
Generally, these discrepancies can be established by proving that  $\rmtlD^t(0,nx)$ and $\rmtlT^t_M(0,nx)$ have the same geodesics with overwhelming probability. To justify this, we introduce the notion of effective dynamic radius $(\re)_{e \in \cE}$
in Subsection \ref{Subsection: effective radius}. Roughly speaking, for an edge $e$ lying on a geodesic of $\rmtlT^t_M(0,nx)$, the quantity $\re$ controls the length of
an open bypass around $e$, and hence quantifies the impact of resampling $\tau_e(t)$ on $\rmtlT^t_M(0,nx)$. Moreover, these effective dynamic radii possess two useful properties (see Proposition \ref{Propo: good properties of re}):
 \begin{itemize}
     \item [(a)] First, they are relatively local dependent, i.e. there exists a positive constant $C_*= C_*(p,d)$ such that the event $\{\re \leq \ell \}$ depends only on the weight of edges inside box $\Lambda_{C_*\ell}(e)$.  
     \item[(b)]  Second, they satisfy an exponentially small
tail probability : $\pp(\re \geq \ell ) \leq \exp(-c\ell)$ for all $1 \leq \ell \leq \exp(c M)$. 
 \end{itemize}
On the event \{$\forall \gamma \in \G^t_M([0]_t,[nx]_t), \forall e \in \gamma: \re \leq M^{3/4}$\}, which holds with overwhelming probability, we deduce that any geodesic $\gamma$ of $\rmtlT^t_M(0,nx)$ is open (i.e. consists only of $1$-weight edges) and is therefore also a geodesic of $\rmtlD^t(0,nx)$. 
\\
 \noindent \underline{\textbf{Step 2}}: \textbf{Dynamical variance inequalities for first truncated passage time}. Our aim is to prove the following:
\begin{theorem}\label{Theorem: a dynamic formula for Dkm} There exists some constant $C > 0$ such that for all $n >1$ and $x \in \Z^d$,
\begin{align}
 \frac{1}{C} \int_0^1 \E[|\tlpim(0,nx) \cap \tlpimt(0,nx)|] dt -C (\log n)^{d+8} \leq  \var(\rmtlTm(0,nx)),
\end{align}
and for all $\alpha > 0$,
\begin{align}
   \var(\rmtlTm(0,nx)) \leq C \alpha^2 \int_0^1 \E[|\tlpim(0,nx) \cap \tlpimt(0,nx)|] dt + C\exp(-\alpha/C) n + C (\log n)^{2d+6}, 
\end{align}
where the function $\E[|\tlpim(0,nx) \cap \tlpimt(0,nx)|]$ is non-increasing in $t$.
\end{theorem}
 %%%%%%%%%%%%%%%%%%%%%%%%%%%%%%%%%%%%%%%%%%%%
 Our proof is inspired by the covariance representation for first passage percolation with edge-weight distributions having finite second moment, developed in \cite{ahlberg2023chaos}. A related formulation in the Gaussian setting is central  to Chatterjee's work in establishing the first rigorous connection between superconcentration and chaos  \cite{chatterjee2014superconcentration}. Analogous covariance formulas for Boolean functions also play an important role in the study of noise sensitivity of critical percolation \cite{tassion2023noise}. In the present work, we adapt this representation to analyze co-influences in the setting where the edge‐weight distribution has unbounded support, allowing infinite values. More precisely, we relate the dynamical covariance to the total co-influence at time
$t$ as follows. For every $t \in [0,1]$, we have
    \begin{align}\label{Eq: Representation of cov of D}
      \Cov(\rmtlD(0,nx),\rmtlD^t(0,nx))= \int_t^1 \sum_{e \in \cE}  \E\left[\Delta_e(\rmtlD,\rmtlD^s)\right]ds.
    \end{align}
As a consequence, we get
    \begin{align}\label{Eq: Representation of squre of D-Dt}
        \E[(\rmtlD(0,nx)-\rmtlD^t(0,nx))^2]= 2\int_0^t \sum_{e \in \cE} \E\left[\Delta_e(\rmtlD,\rmtlD^s)\right] ds.
\end{align}
Similarly, 
      \begin{align} \label{Eq: formula of cov of Tm}
\Cov(\rmtlTm(0,nx),\rmtlTmt(0,nx))= \int_t^1\sum_{e \in \cE}  \E\left[\Delta_e(\rmtlTm,\rmtlT^s_M)\right]ds,
    \end{align}
where by a simple computation, we can rewrite
\begin{align} \label{Eq: co-influe and co-impact}
\E\left[\Delta_e(\rmtlTm,\rmtlT^t_M)\right] = p(1-p)\E\left[ \influ_e(\rmtlTm,\rmtlT^t_M) \right], \quad  \influ_e(\rmtlTm,\rmtlT^t_M) := \nabla^{M,1}_e \rmtlT_M(0,nx) \nabla^{M,1}_e \rmtlTmt(0,nx).
\end{align} 
Theorem \ref{Theorem: a dynamic formula for Dkm} can be deduced from \eqref{Eq: formula of cov of Tm} and \eqref{Eq: co-influe and co-impact} together with  the following lower and upper bounds on the total co-influence: there exists a constant $C > 0$ such that for all $t \geq 0$,
   \begin{align}\label{lower bound}
       \sum_{e \in \cE}  \E\left[ \influ_e(\rmtlTm,\rmtlT^t_M) \right]\geq \frac{1}{C}\E[|\tlpim(0,nx) \cap \tlpim^{t}(0,nx)|] - C (\log n)^{d+8}, 
   \end{align}
 and 
\begin{align}\label{upper bound of TM}
\forall \alpha >0, \, \sum_{e \in \cE}\E\left[ \influ_e(\rmtlTm,\rmtlT^t_M)\right]
    \leq C \alpha^2 \E\left[|\tlpim(0,nx) \cap \tlpim^{t}(0,nx)|\right] +  C\exp(-\alpha/C) n + C(\log n)^{2d+6}.
\end{align}
The lower bound \eqref{lower bound} is more straightforward, following from the key observation that $\rmtlT^t_M(0,nx)$ only takes integer values. For the upper bound \eqref{upper bound of TM}, on the event that the regularized points $[0],[nx]$ and $[0]_t, [nx]_t$ are unchanged when closing  the edge $e$, and that $e$ is not close to any of these points, thanks to the tool of effective dynamic radius, we deduce that
\begin{align*}
        0 \leq \influ_e(\rmtlTm,\rmtlT^t_M) \I(\omega_e =\omega_e(t)=1) \leq  (\hre)^2\I(e \in \tlpi_M(0,nx) \cap  \tlpi^t_M(0,nx)),
    \end{align*}
where
\begin{align*}
 \hre := C_*\re \wedge M.
\end{align*}
Thus, we can upper bound the total co-influence at time $t$ as
\begin{align}\label{upper bound}
\sum_{e \in \cE} \E\left[\influ_e(\rmtlTm,\rmtlT^t_M) \right] \leq C \E\left[\sum_{e \in \tlpi_M(0,nx) \cap  \tlpi^t_M(0,nx)} (\hre)^2\right] + C (\log n)^{2d+6}.
\end{align}
Our current goal is to estimate the total cost of edge resampling or the sum
of the effective dynamic radii along the overlap of geodesics (RHS of \eqref{upper bound}). Thanks to the properties (a) and (b) of effective dynamic radii, applying \textbf{lattice animal theory}, we obtain that
\begin{align*}
  \forall \alpha > 0, \quad   \E\left[\sum_{e \in \tlpi_M(0,nx) \cap  \tlpi^t_M(0,nx)} (\hre)^2 \right] \leq C \alpha^2 \E[|\tlpi_M(0,nx) \cap  \tlpi^t_M(0,nx)|] + C\exp(-\alpha/C) n.
\end{align*}
\subsection{Proof of Theorem \ref{Theorem: variance and chaos} and \ref{theorem: main} via Theorem \ref{Theorem: a dynamic formula for Dkm}} \label{Section: proof of main results}
\begin{proof}[Proof of Theorem \ref{Theorem: variance and chaos}]
    For (i), using Theorem \ref{Theorem: a dynamic formula for Dkm} and \eqref{Eq: variance difference}, \eqref{Eq: gedesics size difference}, we get
    \begin{align*}
        \var(\rmtlD(0,nx)) & \geq   \var(\rmtlT_M(0,nx)) - \cO(\exp(-M^{3/4})) \\
        & \geq c \int_0^1 \E[|\tlpi_M(0,nx) \cap \tlpi_M^{t}(0,nx)|] dt- \cO((\log n)^{d+8}) \\
        & \geq  c \int_0^1 \E[|\tlpi(0,nx) \cap \tlpi^{t}(0,nx)|] dt - \cO((\log n)^{d+8}),
    \end{align*}
for some constant $ c >0$. For (ii), it follows from Theorem \ref{Theorem: a dynamic formula for Dkm} and \eqref{Eq: variance difference}, \eqref{Eq: gedesics size difference} that there exists a constant $C>0$ such that for all $\alpha > 0$,
\begin{align*}
        \var(\rmtlD(0,nx)) & \leq   \var(\rmtlT_M(0,nx)) + \cO(\exp(-M^{3/4})) \\
        & \leq  C \alpha^2 \int_0^1 \E[|\tlpi_M(0,nx) \cap \tlpi_M^{t}(0,nx)|] dt + C\exp(-\alpha/C) n + C (\log n)^{2d+6} \\
        & \leq  C \alpha^2\int_0^1 \E[|\tlpi(0,nx) \cap \tlpi^{t}(0,nx)|] dt+ C\exp(-\alpha/C) n + C (\log n)^{2d+6}+ \alpha^2 \cO(\exp(-M^{3/4})).
    \end{align*}
Notice that $ \E[|\tlpi(0,nx) \cap \tlpi^{t}(0,nx)|] = \sum_{e \in \cE} \E[\I(e \in \tlpi(0,nx)) \I( e \in \tlpi^{t}(0,nx))]$. Applying Lemma \ref{Lemma: derivative of phi s} to $f = \I(e \in \tlpi(0,nx))$, we have $\E[|\tlpi(0,nx) \cap \tlpi^{t}(0,nx)|]$ is non-increasing in $t$. By Proposition \ref{Thm: intersection of geo}, there exist constants $c', C'>0$ such that for all $t \geq 0$,
\begin{align}\label{Eq: lower bound of intersect pi0,s}
    \E[|\tlpi(0,nx) \cap \tlpi^{t}(0,nx)|] & \geq \E[|\tlpi(0,nx) \cap \tlpi^{1}(0,nx)|]  \geq \sum_{e \in \Lambda_{C' n}} \pr(e \in \tlpi(0,nx), e \in \tlpi^1(0,nx)) \notag \\
    & = \sum_{e \in \Lambda_{C' n}} (\pr(e \in \tlpi(0,nx)))^2  \geq \frac{1}{| \Lambda_{C' n}|}\left(\sum_{e \in \Lambda_{C' n}} \pr(e \in \tlpi(0,nx))\right)^2 \notag \\
    & \geq \frac{1}{| \Lambda_{C' n}|} \left(\E\left[ \sum_{e \in \cE} \I(e \in \tlpi(0,nx)) \I(\tlpi(0,nx) \subset \Lambda_{C' n}) \right] \right)^2 \notag \\
    & \geq \frac{1}{| \Lambda_{C' n}|} (\E[|\tlpi(0,nx)|\I(\tlpi(0,nx) \subset \Lambda_{C' n})] )^2  \notag \\
    &  = \frac{1}{| \Lambda_{C' n}|} (\E[|\tlpi(0,nx)|]-\E[|\tlpi(0,nx)| \I(\tlpi(0,nx)\not \subset \Lambda_{C' n})])^2\notag  \\
    & \geq c'n^{2-d},
\end{align}
where for last inequality we have used that
\begin{align*}
    \E[|\tlpi(0,nx)| \I(\tlpi(0,nx)\not \subset \Lambda_{C' n})] \leq (\E[|\tlpi(0,nx)|^2] \pr(\tlpi(0,nx)\not \subset \Lambda_{C' n}))^{1/2} = o(1),
\end{align*}
as Lemmas \ref{Lem: hole}, \ref{Lem: large deviation of graph distance Dlambda}. Hence, 
\begin{align*}
        \var(\rmtlD(0,nx))  \leq  2C\alpha^2\int_0^1 \E[|\tlpi(0,nx) \cap \tlpi^{t}(0,nx)|] dt+ C\exp(-\alpha/C) n+ C (\log n)^{2d+6},
    \end{align*}
as desired.
\end{proof}
\begin{proof}[Proof of Theorem \ref{theorem: main}]  For stable regime, it follows from \eqref{Eq: comparison of influence}, \eqref{Eq: Representation of squre of D-Dt}, \eqref{Eq: co-influe and co-impact}, \eqref{upper bound of TM} that for $n$ large enough
\begin{align}
      \E\Big[(\rmtlD(0,nx) & -\rmtlD^t(0,nx))^2\Big]  \stackrel{\eqref{Eq: Representation of squre of D-Dt}}{=} 2\int_0^t \sum_{e \in \cE} \E\left[\Delta_e(\rmtlD,\rmtlD^s) \right] ds \notag \\
      & \stackrel{\eqref{Eq: comparison of influence}}{\leq} 2\int_0^t \sum_{e \in \cE}  \E\left[ \Delta_e(\rmtlT_M\rmtlT^s_M)\right] ds+t \cO(\exp(-M^{3/4})) \notag \\
      &  \stackrel{\eqref{Eq: co-influe and co-impact}, \eqref{upper bound of TM}}{\leq} 2 C\alpha^2\int_0^t \E[|\tlpi_M(0,nx) \cap \tlpi^{s}_M(0,nx)|]ds +2 C\exp(-\alpha/C) t n + 2 C t (\log n)^{2d+6}+ t \cO(\exp(-M^{3/4})) \notag \\
      & \qquad  \leq  C't n,
\end{align}
for some constant $C'>0$. Here for the last line we have used that $\E[|\tlpi_M(0,nx)|] \leq \E[\rmtlT_M(0,nx)] = \cO(n)$. Therefore,
\begin{align}
  \corr(\rmtlD(0,nx),\rmtlD^t(0,nx))  = \frac{\Cov(\rmtlD(0,nx),\rmtlD^t(0,nx))}{\var(\rmtlD(0,nx))}   & = 1 - \frac{\E[(\rmtlD(0,nx)-\rmtlD^t(0,nx))^2]}{2\var(\rmtlD(0,nx))} \notag\\
  & \geq 1 - \frac{C'nt}{\var(\rmtlD(0,nx))}  \geq 1 -C' \beta,
\end{align}
holds for any sequence $t:=t_n$ such that $\frac{t_n} {\hat{t}_n} \leq \beta$. We next establish the chaotic regime for "large" $t$. Thanks to \eqref{Eq: gedesics size difference}, \eqref{Eq: comparison of influence}, \eqref{Eq: Representation of cov of D}, \eqref{Eq: co-influe and co-impact}, \eqref{lower bound},
we have 
\begin{align*}
        \var(\rmtlD(0,nx)) \stackrel{\eqref{Eq: Representation of cov of D}}{=} \int_0^1 &  \sum_{e \in \cE} \E\left[\Delta_e(\rmtlD,\rmtlD^t) \right] dt 
         \stackrel{\eqref{Eq: comparison of influence}}{\geq} \int_0^t\sum_{e \in \cE} \E\left[\Delta_e(\rmtlT,\rmtlT^s_M) \right] ds - t\cO(\exp(-M^{3/4})) \\
        & \stackrel{\eqref{Eq: co-influe and co-impact}}{=} p(1-p) \int_0^t\sum_{e \in \cE}  \E\left[ \influ_e(\rmtlTm,\rmtlT^s_M)\right]ds  - t\cO(\exp(-M^{3/4})) \\
        & \stackrel{\eqref{lower bound}}{\geq} c \int_0^t \E[|\tlpim(0,nx) \cap \tlpim^{s}(0,nx)|] ds - t \cO((\log n)^{d+8}) \\
        & \stackrel{\eqref{Eq: gedesics size difference}}{\geq} c \int_0^t \E[|\tlpi(0,nx) \cap \tlpi^{s}(0,nx)|] ds -t \cO((\log n)^{d+8})\\
         & \geq  c t \E[|\tlpi(0,nx) \cap \tlpi^{t}(0,nx)|] - t \cO((\log n)^{d+8}).
\end{align*}
This implies that there exists a constant $c>0$ such that
\begin{align*}
\E[|\tlpi(0,nx) \cap \tlpi^{t}(0,nx)|]
& \leq c^{-1}(\var(\rmtlD(0,nx))/t + (\log n)^{d+8})  \leq c^{-1} (n/\beta + (\log n)^{d+8}), 
\end{align*}
holds for any sequence $t:=t_n$ such that $ \frac{t_n} {\hat{t}_n} \geq \beta$.
\end{proof}
\subsection{Structure of the paper and some notations}
The article is organized as follows.
In Section~\ref{Sec: 2}, we introduce the main ingredients of the proofs.
Sections~\ref{Sec: 3} and~\ref{Sec: 4} are devoted to the proofs of Step~1 and Step~2 using the elements prepared in Section~\ref{Sec: 2}.
In the Appendix, we provide the proofs of the dynamical formula for covariance and the lower bound on the intersection of geodesics. To conclude this section, we introduce some notations that will be used throughout the paper.
 \begin{itemize}
   \item \emph{Integer interval}. Given an integer $ \ell \geq 1$, we denote by $[\ell]:= \{1, 2,\ldots,\ell \}$.
  \item \emph{Metric}. We denote by $\|\cdot \|_{1},\|\cdot \|_{\infty}, \|\cdot\|_{2}$  the $l_1,l_{\infty},l_{2}$ norms, respectively.  
  \item \emph{Box and its boundary}. Let $x \in \Z^d$ and $\ell >0$, we will denote by $\Lambda_\ell(x):= x+ [-\ell, \ell]^d \cap \Z^d$ the box centered  at $x =(x_1,\ldots,x_d) \in \Z^d$ with radius $\ell$. For convenience, we briefly write $\Lambda_\ell = [-\ell, \ell]^d $ for $\Lambda_\ell(0)$. We define the boundary of $\Lambda_\ell(x)$ as $\partial \Lambda_\ell(x):=\Lambda_\ell(x) \setminus \Lambda_{\ell-1}(x)$.
   \item \emph{$\Z^d$-path and set of $\Z^d$-paths}. For any $\ell \geq 1$, we say that a sequence $\gamma = (v_0,\ldots, v_\ell)$ is a $\Z^d$-path if for all $i \in [\ell],  \|v_i - v_{i-1}\|_1 =1$. The length of $\gamma$ is $\ell$, denoted by $|\gamma|$. For $1 \leq i<j \leq \ell$, we denote by $\gamma_{v_{i},v_{j}}$ the sub-path of $\gamma$ from $v_{i}$ to $v_j$. In addition, if $v_i \neq v_j$ for $i \neq j$, then we say that $\gamma$ is self-avoiding. From now on, we will shortly write a path in place of a self-avoiding $\Z^d$-path. For any path $\gamma$, we denote ${\bf s}(\gamma),{\bf e}(\gamma)$ the starting and ending vertices of $\gamma$, respectively. Given $U \subset \Z^d$, let $\mathcal{P}(U)$ be the set of all paths in $U$.
 \item \emph{Open path, open cluster and crossing cluster}. Given a Bernoulli percolation on $\Z^d$ with parameter $p$, let $\mathcal{G}_p = (\Z^d, \{ e \in \cE: e \text{ is open} \})$. We say that a path is open if all of its edges are open. An open cluster is a maximal connected component of $\mathcal{G}_p$. An open cluster $\kC$ crosses a box $\Lambda$, if for all $d$ directions, there is an open path in $\kC \cap \Lambda$ connecting the two opposite faces of $\Lambda$. 
 \item  \emph{Open connection}. Given $A,B,U \subset \Z^d$, we write $ A \xleftrightarrow{ U} B$ if there exists an open path inside $U$ connecting $A$ to $B$, and otherwise. 
 When $U = \Z^d$, we omit the symbol $\Z^d$ for simplicity.
\item \emph{Diameter}. For $A \subset \Z^d$ and $1 \leq i \leq d$, let us define
$$
\text{diam}_i(A)= \max_{x,y \in A}|x_i-y_i|,
$$
 and we thus denote diam$(A)$ the diameter of $A$ by 
\begin{align*}
     \text{diam}(A) = \max_{1 \leq i \leq d} \text{diam}_i(A).
 \end{align*}
\item  \emph{Geodesic}. Given $\{x,y\} \subset U \subset \Z^d$, we denote by $\pi$ a geodesic between $x$ and $y$ of $\rmD_U(x,y)$ if $
\pi$ is an open path inside $U$ such that $|\pi| = \rmD_U(x,y)$. We also denote by $\eta$ a geodesic between $x$ and $y$  of $\rmT_{U,M}(x,y)$ if $\eta$ is a path inside $U$ such that $\rmT(\eta) = \rmT_{U,M}(x,y)$.
  \end{itemize}
\textit{Throughout this paper, the symbols 
$c,C$ may denote different positive constants whose values may change from line to line.}
\section{Ingredients for the proof of main results}\label{Sec: 2}
\subsection{Preliminary on supercritical percolation} In this subsection, we review some concrete results of chemical distance and crossing cluster of supercritical Bernoulli percolation.
\begin{lemma}\cite[Theorem 2]{pisztora1996surface} \label{Lem: hole}
There exists a constant $c >0$ such that for all $z \in \Z^d$ and $\ell > 0$,
 \begin{align}\label{Claim: 2e0}
      \pr (\|z- [z]\|_{\infty} \geq \ell) \leq c^{-1} \exp (-c \ell^{d-1}). 
 \end{align}
 \end{lemma}
 \begin{lemma} \cite[Theorem 8.18]{grimmett1999percolation} \label{Lemma: 0 not to infty}
      There exists a
constant $c > 0$ such that for all $\ell > 0$,
\begin{align*}
    \pr( 0 \longleftrightarrow \partial \Lambda_{\ell}, 0 \not \longleftrightarrow \infty) \leq c^{-1} \exp(-c \ell).
\end{align*}
 \end{lemma}
\begin{lemma} \cite[Theorem 7.68]{grimmett1999percolation}\label{Lem: pro of crossing event}
There exists a constant $c >0$ such that for all $\ell > 0$,
\begin{align*}
 \pr( \Lambda_{\ell} \text{ has an } \text{open crossing cluster}) \geq 1-c^{-1} \exp(-c \ell).
\end{align*}
\end{lemma}
\begin{lemma} 
\cite[Corollary 2.2, Lemma 2.3]{garet2009moderate}
 \label{Lem: large deviation of graph distance Dlambda}
 There exist $\rho, c > 0$ such that for all $z \in \Z^d$ and  $\ell \geq \rho \|z\|_{\infty} $,
\begin{align}\label{large deviation of D,0-x}
 \max \{ \pr(\infty >\rmD(0,z) \geq  \ell ),\pr(\rmtlD (0,z) \geq  \ell) \} \leq c^{-1} \exp(-c \ell).
\end{align}
Consequently, there exists a constant $C>0$ such that for all $\ell \geq C\|z\|_\infty$,
\begin{align*}
    \max\{\pr(|\gamma| \geq \ell),\pr(|\eta| \geq \ell)\} \leq C\exp(-\ell/C),
\end{align*}
where $\gamma$ is the geodesics of $\rmtlD (0,z) $ and $\eta$ is the union of $\gamma$ and the $\Z^d$-shortest path from $0$ to $[0]$.
\end{lemma}
\begin{lemma} \cite[Lemma 7.104]{grimmett1999percolation} \label{lem: twodisjointclusters}
    For any $\varepsilon > 0$, there exists a constant $c>0$ such that for all $\ell > 0$,
    \begin{align*}
        \pr(\, \exists \, \text{two disjoint open clusters with diameter at least $\varepsilon \ell$ in $\Lambda_\ell$} ) \leq c^{-1} \exp(-c \ell).
    \end{align*}
\end{lemma}
%%%%%%%%%%%%%%%%%%%%%%%%%%%%%%%%%%%%%%%%%%%%%%%%%%
\subsection{The co-influence bound via effective dynamic radius}\label{Subsection: effective radius}
In \cite{can2024subdiffusive}, we introduced the notion of an effective radius to measure the change of the truncated passage time when resampling the state of one single edge. In the present paper, we develop a dynamical version of this notion, aimed at controlling the co-influence of the original truncated passage time $\rmT_M$  and the $t$-noise truncated passage time $\rmT^t_M$ when the weight of a given edge is resampled (see \eqref{Eq: Proposition co-impact}). In essence, effective dynamic radii act as a simple one-step renormalization process, allowing us to construct an open bypass of $e$ with optimal cost. By analyzing the properties of these radii, including the local dependence and light-tailed decay
distribution, we can bound the total cost of resampling edges using lattice animal theory. Let $t \geq 0$. We couple the first truncated passage percolation in original environment $\tau=(\tau_e)_{e \in \cE}$ and in $t$-noise environment $\tau(t)=(\tau_e(t))_{e \in \cE}$. We will call an edge $e$ is $t$-open ($1$-weight) or $t$-closed ($M$-weight) in environment $\tau(t)=(\tau_e(t))_{e \in \cE}$. Let us fix a presentation of edge $e = (x_e,y_e)$ for all $e \in \cE$ by some deterministic rule to breaking ties. Given $N \geq 1$ and $e \in \cE$, we now define $\Lambda_N(e): = \Lambda_N(x_e)$ and   
$$
   \rmA_{N}(e):= \Lambda_{3N}(e) \setminus \Lambda_N(e).
$$
Let us define the set of $t$-open paths in $A\subset \Z^d$ by 
 \begin{align*}
  \mathbb{O}^t{(u,v;A)}:=\{ \gamma \in \kP(A): {\bf s}(\gamma)= u, {\bf e}(\gamma)= v, \textrm{ $\gamma$ is $t$-open}\};\quad \mathbb{O}^t{(A)}:= \bigcup_{u,v \in A}  \mathbb{O}^t{(u,v;A)}.
  \end{align*}
For $u,v \in \Z^d$, and $A \subset \Z^d$, we denote the set of all geodesics of $\rmTmt$ as
$$
\G^t_M(u,v;A):= \{ \gamma \in \kP(A): {\bf s}(\gamma) =u, {\bf e}(\gamma) = v; \rmT_M^t(\gamma) : = \rmT^t_{A,M}(u,v)\},\quad \G^t_M(A):= \cup_{u,v \in A} \G^t_M(u,v;A),
$$
where we recall that ${\bf s}(\gamma)$ and ${\bf e}(\gamma)$ are the starting and ending vertices of $\gamma$, respectively. Also, define the set of modified geodesics by
 \begin{align*}
 \bH^t_M(u,v;A) &:= \Big \{\gamma \in \kP(A): {\bf s}(\gamma)= u, {\bf e}(\gamma)= v, \exists \, \pi \in \G^t_M(A): \gamma \setminus \pi \textrm{ is $t$-open}  \Big\}, \\
  \bH^t_M(A) &:= \bigcup_{u,v \in A} \bH^t_M(u,v;A),
\end{align*}
with the convention that an empty path is $t$-open (particularly, $\sO^t_M(A), \mathbb{O}^t(A) \subset \bH^t_M(A)$). In other words, $\bH^t_M(u,v;A)$ is the collection of paths from $u$ to $v$ inside $A$, formed by substituting segments of a geodesic with $t$-open paths. When $A = \Z^d$, we simply write $\G^t_M(u,v),\G^t_M$ and $\bH^t_M(u,v),\bH^t_M$, respectively. A path $\gamma \subset \rmA_N(e)$ is called a crossing path of $\rmA_N(e)$ if it joints $\partial \Lambda_N(e)$ and $\partial \Lambda_{3N}(e)$.
Let $\sC (\rmA_N(e))$ be the set of all crossing paths of $\rmA_N(e)$. Given $N \geq 1$, $C_* \geq 4$ and $e \in \cE$, let
\begin{align*}
   & \cV^{t}_N(e) := \{ \forall \gamma_1,\gamma_2 \in \bH^{t}_M(\Lambda_{C_*N}(e)) \cap \sC (\rmA_N(e)): \rmD_{\rmA_{N}(e)}^t(\gamma_1,\gamma_2) \leq C_*N \};  \\
    & \mathcal{W}_N^t(e):=\{\forall \,  x,y \in \Lambda_{3N}(e)\text{ with } \rmD^t_{\Lambda_{3N}(e)}(x,y)< \infty: \,\, \rmD^t_{\Lambda_{4N}(e)}(x,y) \leq C_* N\} .
\end{align*}
 For each $e \in \cE$, we now define the \emph{effective dynamic radius} of $e$ as 
\begin{align*}
    \re: = \inf\{ N\geq 1:\cV_N^{0}(e)\cap  \mathcal{W}_N^0(e) \cap \cV_N^{t}(e)  \cap  \mathcal{W}_N^t(e) \text{ occurs}\}.
\end{align*}
Roughly speaking, for each $s \in \{0,t\}$, $\re$ is the smallest radius that enable an $s$-open bypass of $e$ inside $\aAn$ for any crossing paths of $e$ with not many $s$-closed edges. The following proposition tells us some important properties of effective dynamic radius. Moreover, give an edge $e$ in an original path, we could construct upon this path a new path that avoids $e$. 
\begin{proposition}\label{Propo: good properties of re}
There exist constants $C_*=C_*(p,d) \geq 4$ and $c = c(p,d) >0$ such that the following holds for all $t\geq 0$.
\begin{itemize}
    \item [(i)] For all $e \in \cE$ and $\ell \geq 1$, the event $\{\re = \ell\}$ only depends on the status of any edges in $\Lambda_{C_*\ell}(e)$.
    \item [(ii)] For all $e \in \cE$,
\begin{align*}
    \pr(\re \geq \ell) \leq c^{-1} \exp(-c\ell), \quad \forall \ell \leq \exp(cM).
\end{align*}
    \item [(iii)] Let $x,y \in \Z^d$ and $\gamma \in \bH^s_M(x,y)$ for some $s \in \{0,t\}$. Suppose that $e \in \gamma$ such that $x,y \notin \Lambda_{3\re}(e)$. Then there exists another path $\eta_e$ between $x$ and $y$ such that:
    \begin{itemize}
        \item [(iii-a)] $\eta_e \cap \Lambda_{\re}(e) = \varnothing$ and $\eta_e \setminus \gamma $ consists only of $s$-open edges;
        \item [(iii-b)] $ |\eta_e \setminus \gamma| \leq C_* \re$.
    \end{itemize}
\end{itemize}
\end{proposition}
\begin{proof} The proof of this proposition follows from \cite[Proposition 2.6 and 2.7]{can2024subdiffusive}. We observe that $\{\re = \ell\}$ depends solely on the family of events $\{ \cV_N^{0}(e) ,\mathcal{W}_N^0(e),\cV_N^{t}(e), \mathcal{W}_N^t(e)\}_{ 1\leq N \leq \ell}$.
Hence, (i) follows from the fact that both $\cV_N^{t}(e)$ and $\mathcal{W}_N^t(e)$ are measurable with respect to any edge-weights inside $\Lambda_{C_*N}$ for all $t \geq 0$ and $1 \leq N \leq \ell$.

To prove (ii), we observe that 
\begin{align*}
    \pr(\re \geq \ell) & \leq \pr(\{\cV_{\ell-1}^{0}(e)  \cap  \mathcal{W}_{\ell-1}^0(e) \cap \cV_{\ell-1}^{t}(e) \cap  \mathcal{W}_{\ell-1}^t(e)\}^c) \\
    & \leq \pr(\{\cV_{\ell-1}^{0}(e) \cap   \mathcal{W}_{\ell-1}^0(e)\}^c) + \pr(\{\cV_{\ell-1}^{t}(e)  \cap  \mathcal{W}_{\ell-1}^t(e)\}^c).
\end{align*}
It follows the proof of \cite[Proposition 2.6]{can2024subdiffusive} that there exist constants $C_*=C_*(p,d) \geq 4$ and $c = c(p,d) \in (0,1)$ such that for all $t\geq 0$ and $\ell \geq 2$,
\begin{align*}
    \pr(\{\cV_{\ell-1}^{t}(e)  \cap  \mathcal{W}_{\ell-1}^t(e)\}^c) \leq c^{-1}\exp(-c\ell).
\end{align*}
Combining these estimates, we obtain (ii). The proof of (iii) is the same as \cite[Proposition 2.7]{can2024subdiffusive}.
\end{proof}
From now on, we fix the constant $C_*$ as in Proposition \ref{Propo: good properties of re}. Let us now estimate the co-influence of resampling an edge on first truncated passage time at time $t$ via a
truncation of effective dynamic radius. We denote by $\kC_\infty^{t,e}$ the infinite cluster of $\kC^t_\infty \setminus e$, which is almost surely unique. Let $[z]_t^e$ denote the regularized point of $z$ when closing $e$, i.e. $[z]_t^e$ is the closest point to $z$ in $\kC^{t,e}_\infty$, with some deterministic rule breaking ties.
 For short, we simply write $\kC^{0,e}_\infty: =\kC^{e}_\infty$ and $[z]_0^e = [z]^e$. For any $t \geq 0$, we define
\begin{align*}
        \cQ_e^{t}:= \{[0]_t = [0]_t^{e},[nx]_t = ,[nx]_t^{e}\}.
\end{align*}
The event $\cQ_e^t$ indicates that the regularized points $[0]_t$ and $[nx]_t$ are unchanged when closing the edge $e$.
\begin{proposition}\label{Prop: bound influence by re}
On the event $\{\cQ_e^{0} \cap \cQ_e^{t}\}$, we have the following holds.
\begin{itemize}
    \item [(i)] 
\begin{align}\label{EQ: lower bound delta}
      \I(e \in \tlpim(0,nx) \cap \tlpimt(0,nx)) \leq   \influ_e(\rmtlTm,\rmtlT^t_M); 
   \end{align}
   \item [(ii)] 
     \begin{align}\label{Eq: Proposition co-impact}
     0 \leq \influ_e(\rmtlTm,\rmtlT^t_M) \I(\tau_e = \tau_e(t) =1) \leq (3M\I(\cU^0_e \cup \cU^t_e)+(\hre)^2)\I(e \in \tlpim(0,nx) \cap \tlpimt(0,nx)),
    \end{align}
    where for $s \in \{0,t\}$,
    \begin{align*}
       \cU_e^{s}:= \left\{ 3 \re \geq \min(\|e- [0]_{s}\|_{\infty},
\|e-[nx]_{s}\|_\infty)\right\}, \quad  \hre := C_*\re \wedge M.
    \end{align*}
\end{itemize}
\end{proposition}
\begin{proof}
 Suppose that the event $\{\cQ_e^{0} \cap \cQ_e^{t}\}$ occurs. Then for $s \in \{0,t\}$, the regularized points $[0]_s$ and $[nx]_s$ remains unchanged when the edge $e$ is closed ($\tau_e = M$), and so $\nabla^{M,1}_e \rmtlTms(0,nx)$ is non-negative. Notice that $\rmtlTms(0,nx)$ always takes the value in $\N$. Therefore, if $e \in   \tlpi^s_M(0,nx)$ then $e$ is a pivotal edge of $\rmtlTms(0,nx)$, i.e. the value of $\rmtlTms(0,nx)$ significantly modifies when flipping the value of $\tau_e$, and thus for all $s \in \{0,t\}$,
\begin{align*}
    \rmtlTms(0,nx) \circ \sigma^{M}_e \geq \rmtlTms(0,nx) \circ \sigma^{1}_e +1.
\end{align*}
This implies that on the event $\{\cQ_e^{0} \cap \cQ_e^{t}\}$,
\begin{align*}
    \I(e \in \tlpim(0,nx))\leq \nabla^{M,1}_e \rmtlTm(0,nx), \quad  \I(e \in \tlpimt(0,nx))\leq \nabla^{M,1}_e \rmtlTmt(0,nx).
\end{align*}
Multiply both sides of this inequalities, we obtain \eqref{EQ: lower bound delta}. 

 Suppose further that $e$ is $s$-open (i.e. $\tau_e(s) =1$) and $e \notin  \tlpi^s_M(0,nx)$, then closing $e$ ($\tau_e(s) = M$) has no effect on the geodesic, and thus
   $\nabla^{M,1}_e \rmtlTms(0,nx)= 0$. Hence,
\begin{align}\label{Eq: delta1}
 0 \leq  \nabla^{M,1}_e \rmtlTms(0,nx)\I(\tau_e(s) =1) \leq \nabla^{M,1}_e \rmtlTms(0,nx) \I(e \in \tlpi^s_M(0,nx)).
\end{align}
We remain estimate the right hand side of \eqref{Eq: delta1}. If $e\in \tlpi^s_M(0,nx)$, it is clear that
\begin{align}\label{Eq: Bound for delta M,1}
    \nabla^{M,1}_e \rmtlTms(0,nx) \leq M.
\end{align}
Let $\gamma$ be the geodesic of $\rmtlTms(0,nx)$ with some deterministic rules breaking ties. If $ e \in  \tlpi^s_M(0,nx) \subset \gamma$ and $\{\cU^s_{e}\}^c$ occurs, then neither $[0]_s$ nor $[nx]_s$ belongs to $\Lambda_{3\re}(e)$. Applying Proposition \ref{Propo: good properties of re} (iii) to $\gamma \in \sO^s_M([0]_s,[nx]_s) \subset\bH^s_M([0]_s,[nx]_s) $ and $e \in \gamma$, there exists another path $\eta_e$ between $[0]_s$ and $[nx]_s$ such that
\begin{align*}
    \nabla^{M,1}_e \rmtlTms(0,nx) \leq \rmT_M^s(\eta_e) \setminus \gamma = |\eta_e \setminus \gamma| \leq  C_* \re. 
\end{align*}
Combining the last three  estimates, we get for all $s \in \{0,t\},$
 \begin{align*}
        0 \leq
        \nabla^{M,1}_e \rmtlTms(0,nx)\I(\tau_e(s) =1) \leq (\I(\cU^s_e)M+\hre)\I(e \in \tlpi^s_M(0,nx)).
    \end{align*}
Hence, on the event $\{\cQ_e^{0} \cap \cQ_e^{t}\}$, we have
\begin{align*}
        0 \leq
        \nabla^{M,1}_e \rmtlTm(0,nx)\nabla^{M,1}_e & \rmtlTmt(0,nx)\I(\tau_e=\tau_e(t) =1) \\
      & \leq (\I(\cU^0_e) M+ \hre)(\I(\cU^t_e) M+ \hre)\I(e \in \tlpi_M(0,nx)\cap \tlpi^t_M(0,nx))\\
         & \leq (3\I(\cU^0_e \cup \cU^t_e) M^2+ (\hre)^2)\I(e \in \tlpi_M(0,nx)\cap \tlpi^t_M(0,nx)),
    \end{align*}
as required. 
\end{proof}
\subsection{Estimates for lattice animals}
To manage the total cost of edge resampling, we aim to estimate the sum of effective dynamic radii along a random path. Although these radii are not independent, their dependence is relatively local. Using lattice animal theory in dependent environments (see, e.g., \cite{nakajima2019first,damron2020estimates}), we can deduce an upper bound for this sum of these radii.

Let $\mathcal{P}_L$ denote the set of all paths $\gamma$ starting from $0$ satisfying $|\gamma| \leq L$. Given $A >0$ and  $N\in \N$, a collection of Bernoulli random variables $(I_{e,N})_{ e \in \cE}$ is called  $AN$-dependent if for all $e \in \cE$, the variable $I_{e,N}$ is independent of  the random variables $(I_{e',N})_{e' \not \in \kE(\Lambda_{AN}(e))}$.
Each path $\gamma \in \mathcal{P}_L$ is called a lattice animal of size at most $L$. The greedy
lattice animal problem with edge-weights $(I_{e,N})_{e \in \cE}$ seeks to estimate the maximal weight of any path of maximal size at most $L$. Particularly, we define
\begin{align*}
 \forall \gamma \in \kP_L ,\quad  {\rm \Gamma}(\gamma) := \sum_{e\in \gamma}I_{e,N}, \quad {\rm \Gamma}_{L,N} := \max_{\gamma \in \mathcal{P}_L} {\rm \Gamma}(\gamma).
\end{align*}
\begin{lemma} \cite[Lemma 2.6]{nakajima2019first} \label{lemmaxbound}
 Let $A>0$, $N\in \N$ and   a collection of $AN$-dependent Bernoulli random variables $(I_{e,N})_{ e \in \cE}$. There exists a positive constant $C$ depending on $A, d$ such that for all $L \in \N$,
\begin{align*}
\E[{\rm \Gamma}_{L,N}] \leq C L N^{d} q_N^{1/d}, \quad \textrm{where} \quad q_N :=\sup_{e \in \kE(\Z^d)} \E[I_{e,N}].
\end{align*}
\end{lemma}
Consequently, we can control the total weight of an arbitrary random path in terms of its size as follows.
\begin{lemma} \cite[Lemma 3.3]{can2024subdiffusive} \label{Lem: lattice animal} 
Let $A>0$ and $(X_e)_{e \in \cE}$ be a family of non-negative random variables such that for all $e \in \cE$ and $ N \in \N$, 
\begin{align}\label{Pd}
\textrm{the event } \{N-1 \leq X_e < N\} \textrm{ is independent of } (X_{e'})_{e'\in \kE(\Z^d) \setminus \Lambda_{A N}(e))}.
 \end{align} 
We define \(q_N := \sup_{e\in \kE(\Z^d) } \pr(N-1 \leq X_e < N).\) Let $f:[0,\infty) \to [0,\infty)$ be a function satisfying
\begin{align}\label{hd}
 B:=\sum_{N=1}^{\infty} (f_*(N))^2 N^{d+1} q_N^{1/d}< \infty, \quad \textrm{ where} \quad f_*(N):=\sup_{N-1\leq x< N} f(x).
 \end{align} 
 There exists $C=C(A,B)>0$ such that for any $L\geq 1$ and any random paths $\gamma$ starting from $0$ in the same probability space as $(X_e)_{e \in \cE}$,
 \begin{align}\label{Lem: e of sum of fxe}
 \E\left[ \left(\sum_{e \in \gamma} f(X_e) \right)^2\right]  \leq CL^2+ C\sum_{\ell =L} \ell^2 (\pr(|\gamma|=\ell)^{1/2}. 
 \end{align}
 \end{lemma}
%%%%%%%%%%%%%%%%%%%%%%%%%%
Applying Lemma \ref{Lem: lattice animal} with $X_e= \hre$, $A=2C_*$, and $f(x)=x^2$, since the conditions \eqref{Pd} and \eqref{hd} follow from Proposition \ref{Propo: good properties of re} (i) and (ii) respectively, we have the following.
\begin{corollary}\label{corre}
There exists a constant $C>0$ such that for all   $t \geq 0$, $L \geq 1$ and random paths $\gamma$ starting from $0$ in the same probability space as $(\re)_{e \in \kE(\Z^d) }$,
 \begin{align}
 \E\left[ \left(\sum_{e \in \gamma} (\hre)^2\right)^2\right]  \leq CL^2+ C\sum_{\ell =L} \ell^2 (\pr(|\gamma|=\ell)^{1/2}. 
 \end{align}
\end{corollary}
\subsection{Lower bound for the intersection of geodesics}
This subsection is devoted to proving that the size of the intersection of geodesics is of order $n$. 
\begin{proposition}
     \label{Thm: intersection of geo}
Assume that \eqref{Useful assump} holds. Then there exists a constant $c > 0$ such that
\begin{align}
    \E[|\tlpi(0,nx)|] \geq cn.
\end{align}
\end{proposition}
The proof of this result follows block-resampling strategy whose key ingredient is a lower bound on the size of geodesics obtained in \cite[Theorem 1] {jacquet2024strict}. This control relies on the notion of patterns developed in \cite{jacquet2025geodesics}. Roughly speaking, we construct a pattern in which a geodesic is not optimal in terms of
the number of edges and then investigate the number of times it crosses a translate of this pattern. To streamline the exposition, we leave the proofs of Propositions~\ref{Thm: intersection of geo} to the Appendix.
\section{Dynamical variance inequalities for first truncated passage time}\label{Sec: 3}
\subsection{A dynamical formula for variance}
Let $0< \ell < L \in \R_{+} \cup \{\infty\}$. We emphasize that $L$ may be very large and can be $\infty$. For $p \in (p_c(d) ,1)$, we define the distribution $G_p$ on $\{\ell,L\}$ as
\begin{align*}
 G_p := p \delta_\ell + (1-p) \delta_{L},
\end{align*}
where $\delta_L$ stands for the Dirac delta distribution at $L$. Let $E$ be an at most countable index set. For each $e \in E$, let $\sigma_e^a: \{\ell,L\}^{|E|}\to \{\ell,L\}^{|E|}$ denote the operator that changes the value of the coordinate $e$ to $a \in \{\ell,L\}$. Let $X=(X_e)_{e \in E}$ be a family of i.i.d. random variables with the common distribution $G_p$. For any real-valued function $f:  \{\ell,L\}^{|E|} \to \R$, we define the discrete derivative of $f$ at an edge $e$ by
\begin{align*}
    \nabla^{a,b}_e f(X):= f \circ \sigma_e^{a}(X)- f \circ \sigma_e^{b}(X), \quad a,b \in \{\ell,L\},
\end{align*}
which measures the change in the value of $f(X)$ when the coordinate $X_e$ is switched from $a$ to $b$. In \cite{can2023lipschitz}, we used the following Margulis-Russo formula as a key tool to claim the Lipschitz-continuity’s of time constant.
\begin{lemma} \cite[Lemma 2.1]{can2023lipschitz}
\label{lemrusso} Assume that $E$ is finite. Let $f : \{\ell, L\}^{|E|} \to \R$  be a function such that $f(X)$ is integrable. Then, we have
\begin{align*}
 \dfrac{{\rm d} \E[f(X)]}{{\rm d} p} = \sum_{e \in E} \E[\nabla_e^{\ell,L} f(X)], \quad \nabla^{\ell,L}_e f(X) : = f \circ \sigma^\ell_e(X)- f \circ \sigma^L_e(X).
\end{align*}
\end{lemma}
 We now aim to establish a dynamical version of this formula.
Let $X'=(X'_e)_{e \in E}$ be an independent copy of $X=(X_e)_{e \in E}$ with
i.i.d coordinates. Let $U=(U_e)_{e\in E}$ be a family of i.i.d. uniform random variables on the interval $[0,1]$, independent of $X$ and $X'$. Given some $(t_e)_{e \in E} \in [0,1]^{|E|}$, we define $X((t_e)_{e\in E})= (X_e(t_e))_{e \in E}$ the $(t_e)_{e\in E}$-noise weight of $X$ as
\begin{align*}
  \forall e \in E,\quad  X_e(t_e)= \I(U_e \geq t_e) X_e + \I(U_e < t_e) X'_e,
\end{align*}
i.e. $X((t_e)_{e \in E})$ be obtained from $X$ by resampling independently each coordinate $e \in E$ with probability $t_e$. When $t_e = t_{e'} = t$ for all $e,e' \in E$, we simply write $X(t)$, instead of $X(t,\ldots,t)$. Let $f: \{\ell,L\}^{|E|}  \to \R$ be a real-valued function. Let $X^1=(X^1_e)_{e \in E}$ and $X^2=(X^2_e)_{e \in E}$ be two independent families of i.i.d.  random variables with common
distribution $G_p$, and independent of $X$ and $X'$. The following formula provides a representation of the dynamical covariance of a function $f$ in terms of the co-influence of its coordinates.
\begin{lemma}\label{Lemma: covariacne formula}
For any function $f \in L^2(G_p^{|E|})$, we have
    \begin{align} \label{Eq of lem: covariance}
\Cov(f(X),f(X(t))) = \int_{t}^1 \sum_{e \in E} \E\left[\nabla^{X_e,X_e^1}_e f(X) \nabla^{X_e,X_e^2}_e f(X(s))\right] ds.
\end{align}
 Moreover, the covariance $\Cov(f(X),f(X(t)))$ is non-negative and non-increasing with respect to $t \in [0,1]$.
\end{lemma}
We leave the proof of this lemmas in Appendix \ref{Appendix: DF for variance}. Now we give an explicit formula of the expectation in \eqref{Eq of lem: covariance}.
%%%%%%%%%%%%%%%%555
\begin{lemma}\label{Lem: Efx1x2=El1}
For any function $f \in L^2(G_p^{|E|})$, we have
    \begin{align*}
        \forall e \in E,\quad  \E\left[\nabla_e^{X_e,X^1_e} f(X) \nabla^{X_e,X^2_e}_e f(X(t))\right]  =  p(1-p)\E\left[\nabla_e^{L,\ell} f(X) \nabla^{L,\ell}_e f(X(t))\right].
    \end{align*}
    Moreover, the function $t\to \E\left[\nabla^{X_e,X_e^1}_e f(X) \nabla^{X_e,X_e^2}_e f(X(t))\right]$ is non-negative and non-increasing.
\end{lemma}
\begin{proof}
    By straight computation, we have
    \begin{align*}
        \E &\left[\nabla_e^{X_e,X^1_e} f(X) \nabla^{X_e,X^2_e}_e f(X(t))\right] \\
        & =  \E\left[\nabla_e^{X_e,X^1_e} f(X) \nabla^{X_e,X^2_e}_e f(X(t)) \I(X^1_e > X^2_e)\right]+ \E\left[\nabla_e^{X_e,X^1_e} f(X) \nabla^{X_e,X^2_e}_e f(X(t)) \I(X^1_e < X^2_e)\right]+ \\
        & \qquad \qquad \qquad \qquad \qquad \qquad \qquad \qquad \qquad + \E\left[\nabla_e^{X_e,X^1_e} f(X) \nabla^{X_e,X^2_e}_e f(X(t)) \I(X^1_e = X^2_e)\right] \\
        & = \E\left[\nabla_e^{X_e,L} f(X) \nabla^{X_e,\ell}_e f(X(t))  \I(X^1_e =\ell, X^2_e=L)\right] + \E\left[\nabla_e^{X_e,\ell} f(X) \nabla^{X_e,L}_e f(X(t))\I(X^1_e =L, X^2_e=\ell)\right] \\
        & \qquad \qquad \qquad \qquad \qquad \qquad \qquad \qquad \qquad + \E\left[\nabla_e^{X_e,X^1_e} f(X) \nabla^{X_e,X^2_e}_e f(X(t)) \I(X^1_e = X^2_e)\right]. 
    \end{align*}
Since $X_e$ takes only two values $\ell$ and $L$, we have the first and second terms are zero. For the third term, we rewrite
\begin{align*}
    & \E  \left[\nabla_e^{X_e,X^1_e} f(X) \nabla^{X_e,X^2_e}_e f(X(t)) \I(X^1_e = X^2_e)\right] \\
    & = 
    \E\left[\nabla_e^{X_e,\ell} f(X) \nabla^{X_e,\ell}_e f(X(t)) \I(X^1_e = X^2_e=\ell)\right]  + \E\left[\nabla_e^{X_e,L} f(X) \nabla^{X_e,L}_e f(X(t))\I(X^1_e = X^2_e=L)\right] \\
    & = \E\left[\nabla_e^{L,\ell} f(X) \nabla^{L,\ell}_e f(X(t)) \I(X^1_e = X^2_e = \ell,X_e= L)\right] + \E\left[\nabla_e^{L,\ell} f(X) \nabla^{L,\ell}_e f(X(t)) \I(X^1_e = X^2_e = L,X_e=\ell)\right] \\ 
    & = p(1-p)\E\left[\nabla_e^{L,\ell} f(X) \nabla^{L,\ell}_e f(X(t))\right].
\end{align*}
Notice that the function $t\to \E\left[\nabla_e^{L,\ell} f(X) \nabla^{L,\ell}_e f(X(t))\right]$ is non-increasing follows by applying Lemma \ref{Lemma: derivative of phi s} to $h = \nabla_e^{L,\ell}$.
The result follows.
\end{proof}
Combining Lemma \ref{Lemma: covariacne formula} with Lemma \ref{Lem: Efx1x2=El1} give us the dynamical formula of covariance as follows.
\begin{proposition} \label{covariance formula for f}
 For any function $f \in L^2(G_p^{|E|})$, we have for any $t \geq 0$,
    \begin{align*}
        \Cov(f(X),f(X(t)))= p(1-p)\int_t^1 \sum_{e \in E} \E\left[\nabla_e^{L,\ell} f(X) \nabla^{L,\ell}_e f(X(s))\right] ds.
    \end{align*}
\end{proposition}
Finally, this proposition implies directly the following dynamical formula of first truncated passage time.
 \begin{proposition}\label{Prop: variance dynamic for Dkm}
 We have for any $t \geq 0$,
    \begin{align}
        \Cov(\rmtlTm(0,nx),\rmtlTmt(0,nx))=  p(1-p)\int_t^1 \sum_{e \in \cE} \E\left[\influ_e(\rmtlTm,\rmtlT^s_M)\right] ds.
    \end{align}
\end{proposition}
%%%%%%%%%%%%%%%%%%%%%%%%%%%
\subsection{Proof of Theorem \ref{Theorem: a dynamic formula for Dkm}} 
In view of Proposition \ref{Prop: variance dynamic for Dkm}, Theorem \ref{Theorem: a dynamic formula for Dkm} will follow directly from the following lower and upper bounds on the total co-influence.
\begin{proposition}\label{Prop: bound for influence} There exists a constant $C > 0$ such that for all $t\geq 0$,
\begin{align}\label{Eq: lower bound}
\E\left[|\tlpim(0,nx) \cap \tlpim^{t}(0,nx)|\right] - C\log^{d+8} n \leq \sum_{e \in \cE} \E\left[\influ_e(\rmtlTm,\rmtlT^t_M)\right],
\end{align}
and for all $\alpha >0$,
\begin{align}\label{Eq: upper bound}
    \sum_{e \in \cE} \E\left[\influ_e(\rmtlTm,\rmtlT^t_M)\right]
    \leq C \alpha^2  \E\left[|\tlpim(0,nx) \cap \tlpim^{t}(0,nx)|\right] + Ce^{-\frac{\alpha}{C}} n + C (\log n)^{2d+6}.
\end{align}
Moreover, the function $\E\left[|\tlpim(0,nx) \cap \tlpim^{t}(0,nx)|\right]$ is non-increasing in $t$.
\end{proposition}
The rest of this subsection is devoted to the proof of Proposition \ref{Prop: bound for influence}. We begin with a lemma that provides an upper bound for the total co-influence of the first truncated passage time with respect to the edges that modify the regularized points.
\begin{lemma}\label{Lemma: total sum of influence of D and T}
    \begin{align}\label{Eq of Lemma: bound influecen}
         \sum_{e \in \cE}  \E\left[ |\influ_e(\rmtlTm,\rmtlT^t_M)|\I(\{\cQ^0_e \cap \cQ^t_e\}^c)\right]= \cO(\log^{d+8} n).
    \end{align}
\end{lemma}
\begin{proof} We observe that
\begin{align}
    \E  \left[ |\influ_e(\rmtlTm,\rmtlTmt)| \I(\{\cQ^0_e \cap \cQ^t_e\}^c)\right] & \leq \E\left[ |\influ_e(\rmtlTm,\rmtlTmt)| \I(\{\cQ^0_e\}^c \cap \cQ^t_e )\right]  + \E\left[ |\influ_e(\rmtlTm,\rmtlTmt)|  \I(\cQ^0_e \cap \{\cQ^t_e\}^c )\right]  \notag \\
   & \qquad \qquad \qquad +2 \E\left[ |\influ_e(\rmtlTm,\rmtlTmt)|  \I(\{\cQ^0_e\}^c \cap \{\cQ^t_e\}^c )\right]. 
\end{align}
By \eqref{Eq: delta1} and \eqref{Eq: Bound for delta M,1}, we have
\begin{align*}
    \E\left[ |\influ_e(\rmtlTm,\rmtlTmt)| \I(\{\cQ^0_e\}^c \cap \cQ^t_e )\right] \leq M \E\left[ |\nabla^{M,1}_e \rmtlTm| \I(\{\cQ^0_e\}^c)\right],
\end{align*}
and 
\begin{align*}
 \E\left[ |\influ_e(\rmtlTm,\rmtlTmt)|  \I(\cQ^0_e \cap \{\cQ^t_e\}^c) \right] & \leq M\E\left[ |\nabla^{M,1}_e \rmtlTmt|\I(\{\cQ^t_e\}^c)\right] = M \E\left[ |\nabla^{M,1}_e \rmtlTm| \I(\{\cQ^0_e\}^c)\right].
\end{align*}
Moreover, thanks to Cauchy-Schwarz inequality, we have
\begin{align*}
    \E  \left[ |\influ_e(\rmtlTm,\rmtlTmt)| \I(\{\cQ^0_e\}^c \cap \{\cQ^t_e\}^c )\right]  &\leq \left(\E\left[ |\nabla^{M,1}_e \rmtlTm|^2  \I(\{\cQ^0_e\}^c)\right] \E\left[ |\nabla^{M,1}_e \rmtlTmt|^2 \I(\{\cQ^t_e\}^c)\right]\right)^{1/2} \\
    & = \E\left[ |\nabla^{M,1}_e \rmtlTm|^2\I(\{\cQ^0_e\}^c)\right].
\end{align*}
Notice that if $\tau _e = M$ then closing the edge $e$ does not affect the regularized points, and thus
$$\I(\{\cQ^0_e\}^c) = \I(\{\cQ^0_e\}^c) \I(\tau_e =1) + \I(\{\cQ^0_e\}^c) \I(\tau_e =M) = \I(\{\cQ^0_e\}^c) \I(\tau_e =1).$$
Putting this together with the above bounds, we get
\begin{align} \label{Eq: bounds by QeC}
    \E  \left[ |\influ_e(\rmtlTm,\rmtlTmt)| \I(\{\cQ^0_e \cap \cQ^t_e\}^c)\right] \leq 4 M \E\left[ |\nabla^{M,1}_e \rmtlTm|^2 \I(\tau_e = 1)\I(\{\cQ^0_e\}^c)\right].
\end{align}
To proceed the right hand side, we define 
\begin{align*}
    \kappa_e:=\rmT_{M} ([0]^e,[nx]^e) -\rmT_{M}([0],[nx]).
\end{align*}
Since $ \nabla^{M,1}_e \rmtlTm \I(\tau_e =1) =  \kappa_e \I(\tau_e =1)$, we have
\begin{align} \label{Eq: bounds by kapae}
    \E\left[ |\nabla^{M,1}_e \rmtlTm|^2  \I(\tau_e =1) \I(\{\cQ^0_e\}^c)\right] \leq  \E[ (\kappa_e)^2 \I(\{\cQ^0_e\}^c)].
\end{align} 
By independence property,
\begin{align*}
    \E[(\kappa_e)^4] \leq 8 (\E[(\rmT_{M} ([0]^e,[nx]^e))^4] + \E[(\rmT_{M}([0],[nx]))^4])  & \leq  8 (\E[(\rmD ([0]^e,[nx]^e))^4] +8 \E[(\rmD([0],[nx]))^4]) \\
   & \leq  \frac{8}{1-p} \E[(\rmD ([0]^e,[nx]^e))^4\I(\omega_e =\infty)] + 8\E[(\rmD([0],[nx]))^4] \\
    &  \leq 8\frac{2-p}{1-p}n^4.
\end{align*}
If $\{[0] \neq [0]^e\}$ (resp. $\{[nx] \neq [nx]^e\}$), then $[0]$ (resp. $[nx]$)  is in a finite open cluster of $x_e$ or $y_e$ in $\cG_p\setminus e$, which we denote by $\kC([0],e)$ (resp. $\kC([nx],e)$).
By Lemma \ref{Lemma: 0 not to infty},
\begin{align} \label{Eq: probability of Q^c}
    \pr(\{\cQ^{0}_e\}^c) & \leq \pr([0] \neq [0]^e) + \pr([nx] \neq [nx]^e) \notag \\
    & \leq \pr([0] \not \in \Lambda_{\|e\|_{\infty}/2}) + \pr(\infty>|\kC([0],e)| \geq \|e\|_{\infty}/2) \notag \\
    & \qquad \qquad \qquad \qquad + \pr([nx] \not \in \Lambda_{\|e-nx\|_{\infty}/2}(nx)) + \pr(\infty > |\kC([nx],e)| \geq \|e-nx\|_{\infty}/2) \notag \\
    & \leq c^{-1} \exp(-c\min(\|e\|_\infty,\|e-nx\|_{\infty}),
\end{align}
for some small constant $c>0$. If $\min(\|e\|_\infty,\|e-nx\|_{\infty}) \geq 10 (\log n)/c$, then for $n$ large enough
\begin{align}\label{Eq: >=Clog n}
     \E[ (\kappa_e)^2 \I(\{\cQ^0_e\}^c)] \leq \exp(-c\min(\|e\|_\infty,\|e-nx\|_{\infty})/2).
\end{align}
Assume that $\min(\|e\|_\infty,\|e-nx\|_{\infty}) \leq 10 (\log n)/c$. We suppose further that $\|e-nx\|_{\infty} \leq 10 (\log n)/c$.  The remain
case $\|e\|_\infty \leq 10 (\log n)/c$ can be treated similarly. Thanks to triangle inequality,
\begin{align*}
  \rmT_{M} ([0],x_e) - \rmT_{M} (x_e,[nx]^e) \leq  \rmT_{M} ([0],[nx]^e)) \leq \rmT_{M} ([0],x_e) + \rmT_{M} (x_e,[nx]^e).
\end{align*}
Note that $\rmT_{M}([0],[nx]) =  \rmT_{M} ([0],x_e) + \rmT_{M} (x_e,[nx])$. Hence, we have
\begin{align*}
    \E[ (\kappa_e)^2 \I([0] = [0]^e, [nx] \neq [nx]^e)] 
    & \leq \E[(\rmT_{M} (x_e,[nx]^e)+ \rmT_{M} (x_e,[nx]))^2 ] \\
    & \leq 2\frac{1}{1-p}\E[(\rmT_{M}(x_e,[nx]))^2\I(\tau_e =M)]+ 2(\E[(\rmT_{M} (x_e,[nx]))^2] \\
    & \leq \frac{2(2-p)}{1-p} \E[(\rmT_{M} (x_e,[nx]))^2] \\
    & \leq \frac{4(2-p)}{1-p} (\E[(\rmT_{M} (x_e,nx))^2] + \E[(\rmT_{M} (nx,[nx]))^2]) = \cO((\log n)^6),
\end{align*}
which implies that if $\|e -nx\|_\infty \leq 10 (\log n)/c$ then by \eqref{Eq: probability of Q^c},
\begin{align*}
    \E[ (\kappa_e)^2 \I(\{\cQ^0_e\}^c)] & \leq  \E[ (\kappa_e)^2 \I([0] \neq [0]^e)] +  \E[ (\kappa_e)^2 \I([0] = [0]^e, [nx] \neq [nx]^e)] =  \cO((\log n)^6).
\end{align*}
Combining this with \eqref{Eq: >=Clog n} yields that
\begin{align*}
    \sum_{e \in \cE} \E[ (\kappa_e)^2 \I(\{\cQ^0_e\}^c)]
    & \leq \sum_{e: \min(\|e\|_\infty,\|e-nx\|_{\infty}) \geq 10 (\log n)/c} \exp(-c\min(\|e\|_\infty,\|e-nx\|_{\infty})/2) \\
    & \qquad \qquad \qquad + \sum_{e: \min(\|e\|_\infty,\|e-nx\|_{\infty}) \leq 10 (\log n)/c } \cO((\log n)^6) \\
    & = \cO((\log n)^{d+6}).
\end{align*}
Putting this estimate together \eqref{Eq: bounds by kapae} and \eqref{Eq: bounds by QeC}, we obtain the desired result.
\end{proof}
\begin{proof}[Proof of Proposition \ref{Prop: bound for influence}]
For convenience, we fix $t \geq 0$ and set $\eta^{0,t} :=\tlpi_M(0,nx) \cap \tlpi^t_M(0,nx) $ the random subset of edges.
%%%%%%%%%%%%%%%%%%%%%%%%%%%%%%
For lower bound \eqref{Eq: lower bound}, using Proposition \ref{Prop: bound influence by re} (i), we get
\begin{align*}
  \sum_{e \in \cE}  \E\left[\influ_e(\rmtlTm,\rmtlT^t_M)\I(\cQ^{0}_e \cap \cQ^{t}_e )\right] & \geq  \sum_{e \in \cE}  \E\left[\I(e \in \eta^{0,t})\I(\cQ^{0}_e \cap \cQ^{t}_e )\right]  \\
  & = \sum_{e \in \cE} \E[\I(e \in \eta^{0,t})(1- \I(\{\cQ^{0}_e \cap \cQ^{t}_e)\}^c)]  \\
  & \geq \E[|\eta^{0,t}|]- \sum_{e \in \cE}   \pr( \{\cQ^{0}_e\}^c \cup \{\cQ^{t}_e\}^c).
\end{align*}
It follows from \eqref{Eq: probability of Q^c} that there exists a constant $c>0$ such that for all $s \in \{0,t\}$,
\begin{align*}
    \pr(\{\cQ^{s}_e\}^c) 
     \leq c^{-1} \exp(-c\min(\|e\|_\infty,\|e-nx\|_{\infty}).
\end{align*}
Combining these estimates with Lemma \ref{Lemma: total sum of influence of D and T} yields that
\begin{align*}
    \sum_{e \in \cE}  \E\left[\influ_e(\rmtlTm,\rmtlT^t_M)\right] & \geq  \E[|\eta^{0,t}|]- \sum_{e \in \eta^{0,t}} \pr( \{\cQ^{0}_e\}^c \cup \{\cQ^{t}_e\}^c)- \sum_{e \in \cE}  \E\left[|\influ_e(\rmtlTm,\rmtlT^t_M)|\I(\{\cQ^{0}_e \cap \cQ^{t}_e \}^c)\right] \\
    & \geq \E[|\eta^{0,t}|] - \cO(\log^{d+8} n).
\end{align*}
For upper bound \eqref{Eq: upper bound}, we first write from the independence property that 
\begin{align} \label{Eq: devide influence}
   \sum_{e \in \cE}  & \E\left[\influ_e(\rmtlTm,\rmtlT^t_M) \right] \leq \frac{1}{p^2}\sum_{e \in \cE} \E\left[\influ_e(\rmtlTm,\rmtlT^t_M)\I(\tau_e =\tau_e(t)=1)\right] \notag \\
   &  \leq  \frac{1}{p^2}\left( \sum_{e \in \cE} \E\left[\influ_e(\rmtlTm,\rmtlT^t_M)\I(\tau_e =\tau_e(t)=1) \I(\cQ^{0}_e \cap \cQ^{t}_e)\right] +\sum_{e \in \cE} \E\left[|\influ_e(\rmtlTm,\rmtlT^t_M)|\I(\{\cQ_e^0 \cap \cQ_e^t\}^c)\right] \right).
\end{align}
From Proposition \ref{Prop: bound influence by re} (ii), it follows that for all $e \in \cE$,
\begin{align} \label{Eq: co-impact}
 0  \leq  \E & \left[\influ_e(\rmtlTm,\rmtlT^t_M) \I(\tau_e =\tau_e(t)=1) \I(\cQ^{0}_e \cap \cQ^{t}_e)\right]   \leq \E\left[ (3\I(\cU^{0}_e \cup \cU^{t}_e ) M^2+ (\hre)^2)\I(e \in \eta^{0,t})\right] \notag \\
 & \qquad \qquad \qquad \leq 3M^2\E\left[\sum_{e \in \eta^{0,t}}\I(\cU^{0}_e)\right]   +  3M^2\E\left[\sum_{e \in \eta^{0,t}}\I(\cU^{t}_e)\right]+ \E\left[\sum_{e \in \eta^{0,t}} (\hre)^2\right].
\end{align}
For $s \in \{0,t\}$, set $r^s_{e}: = \min(\|e- [0]_{s}\|_{\infty},
\|e-[nx]_{s}\|_\infty)$. If $r^s_{e} \geq 3M$, then $\re \geq M$ and so $\hre = M$. Therefore, we have
\begin{align}\label{Eq: bound of double gradient}
   \sum_{e \in \cE} & \E \left[\influ_e(\rmtlT,\rmtlT^t_M) \I(\tau_e =\tau_e(t)=1) \I(\cQ^{0}_e \cap \cQ^{t}_e)\right]   \notag \\
     & \leq  3M^2\E\left[\sum_{e \in \eta^{0,t}}\I(r^0_e \leq 3M)\right]   +  3M^2\E\left[\sum_{e \in \eta^{0,t}}\I(r^t_e \leq 3M)\right]+  7 \E\left[\sum_{e \in \eta^{0,t}} (\hre)^2\right].
\end{align}
It is clear that for all $ s \in \{0,t\}$,
\begin{align}\label{Eq: sum of e in eta0t of I(re<3M}
\sum_{e \in \eta^{0,t}}\I(r_e^s \leq 3M)= |e \in \eta^{0,t}: r_e^s \leq 3M| & \leq   |\{e \in \eta^{0,t}: \|[0]_{s}-e\|_\infty \leq 3M\}| + |\{e \in \eta^{0,t}: \|[nx]_{s}-e\|_\infty \leq 3M\}| \notag \\
& = \cO(M^d).
\end{align}
Thus, we remain to control the third term of the right hand side of \eqref{Eq: bound of double gradient}. Observe that for any $\alpha> 0$,
\begin{align} \label{Eq: sum of e in eta0t}
    \E\left[\sum_{e \in \eta^{0,t}} (\hre)^2\right] & = \E\left[\sum_{e \in \eta^{0,t}} (\hre)^2 \I(\hre \leq \alpha) \right] +  \E\left[\sum_{e \in \eta^{0,t}} (\hre)^2 \I( \hre > \alpha) \right] \notag \\
    & \leq \alpha^2 \E\left[|\eta^{0,t}|\right] +  \E\left[\sum_{e \in \eta^{0,t}} (\hre)^2 \I(\hre \geq \alpha)\right].
\end{align}
Let $\gamma^1$ be the geodesic of $\rmtlTm(0,nx)$ for some deterministic rule breaking tie, so that $\eta^{0,t} \subset \gamma^1$. Let $\gamma^2$ be the shortest $\Z^d$-path from $0$ to $[0]$. Let $\gamma: = \gamma^1 \cup \gamma^2$ be the random path starting $0$. Notice that $|\gamma^2| = \|[0]-0\|_1 \leq d\|[0]-0\|_\infty$. Therefore, by Lemmas \ref{Lem: hole} and \ref{Lem: large deviation of graph distance Dlambda},
\begin{align}\label{Eq: lagre deviation of |eta|}
    \pr(|\gamma| \geq 2d C_1 n) \leq \pr( |\gamma^1| \geq dC_1 n) + \pr(|\gamma^2| \geq dC_1n) = c^{-1} \exp(-c n),
\end{align}
for some constant $C_1,c>0$.
Applying Corollary \ref{corre} with $L = 2d C_1 n$, we have for some constant $C>0$,
\begin{align}\label{Eq: sum square of re}
   \E\left[\left(\sum_{e \in \eta^{0,t}} (\hre)^2 \right)^2\right] \leq \E\left[\left(\sum_{e \in \gamma} (\hre)^2 \right)^2\right] \leq C n^2 + \sum_{\ell \geq 2dC_1 n} \ell^2 (\pr(|\gamma| \geq 2dC_1 n))^{1/2} \leq 2Cn^2.
\end{align}
By two above estimates and Cauchy-Schwartz inequality, there exists a constant $c>0$ such that
\begin{align*}
    \E\left[\sum_{e \in \gamma }(\hre)^2  \I(\hre \geq \alpha) \I(|\gamma| \geq 2 dC_1 n)\right]  & \leq \E\left[\sum_{e \in \gamma }(\hre)^2 \I(|\gamma| \geq 2 dC_1 n)\right] \\
    &\leq \left(\E\left[\left(\sum_{e \in \gamma} (\hre)^2 \right)^2\right] \pr\left( |\gamma| \geq 2 dC_1 n\right)\right)^{1/2} \\
   &  \leq c^{-1} \exp(-cn).
\end{align*}
For each $N \geq \lfloor \alpha \rfloor +1$, let $I_{e,N} := \I(N-1 \leq \hre< N)$. Therefore, using Lemma \ref{lemmaxbound}, we obtain that
\begin{align}
    \E\left[\sum_{e \in \gamma } (\hre)^2\I(\hre \geq \alpha) \I(|\gamma| \leq 2 dC_1 n)\right]&  = \sum_{N\geq \lfloor \alpha \rfloor +1} \E\left[\sum_{e \in \gamma} (\hre)^2 I_{e,N} \I(|\gamma| \leq 2 dC_1 n)\right] \notag \\ 
    & \leq \sum_{N\geq \lfloor \alpha \rfloor +1} N^2 \E\left[\sum_{e \in \gamma} I_{e,N} \I(|\gamma| \leq 2 dC_1 n)\right] \notag \\
    & \leq \sum_{N\geq \lfloor \alpha \rfloor +1} N^2 \E\left[\max_{\gamma \in \kP_{2dC_1n}}\sum_{e \in \gamma} I_{e,N}\right] \notag \\
    & \leq 2dC C_1n \sum_{N\geq \lfloor \alpha \rfloor +1} N^{d+2} \exp(-N/(Cd)) \leq C_2n\exp(-\alpha/C_2),
\end{align}
for some constant $C_2=C_2(d,p) >0$. It follows from two above estimates that there exists some constant $C> 0$ such that
\begin{align}
   \E\left[\sum_{e \in \eta^{0,t}} (\hre)^2 \I(\hre \geq \alpha)\right]  \leq Cn\exp(-\alpha/C) + C\exp(-n/C).
\end{align}
Combining this estimate with  \eqref{Eq: sum of e in eta0t} yields that for all $\alpha>0$,
\begin{align}
     \E\left[\sum_{e \in \eta^{0,t}} (\hre)^2\right] \leq \alpha^{2} \E[|\eta^{0,t}|]  + Cn\exp(-\alpha/2C)+ C \exp(-n/C).
\end{align}
 Plugging this and \eqref{Eq: sum of e in eta0t of I(re<3M} into \eqref{Eq: bound of double gradient} concludes that
\begin{align*}
      \sum_{e \in \cE}  \E\left[\influ_e(\rmtlTm,\rmtlT^t_M) \I(\tau_e =\tau_e(t)=1) \I(\cQ^{0}_e \cap \cQ^{t}_e)\right] & \leq C\alpha^{2} \E[|\eta^{0,t}|] + Cn\exp(-\alpha/C) + C (\log n)^{2d+4},
\end{align*}
for some constant $C>0$. Hence, we complete the proof of upper bound by merging this estimate, \eqref{Eq: devide influence}, and Lemma \ref{Lemma: total sum of influence of D and T}.

Finally, we observe that 
\begin{align*}
    \E[|\eta^{0,t}|] = \sum_{e \in \cE} \E[\I(e \in \tlpi_M(0,nx)) \I( e \in \tlpi^t_M(0,nx))].
\end{align*}
By Lemma \ref{Lemma: derivative of phi s}, we have $\E[\I(e \in \tlpi_M(0,nx)) \I( e \in \tlpi^t_M(0,nx))]$ is non-increasing in $t$. Consequently, $\E[|\eta^{0,t}|]$ is non-increasing in $t$ as well.
\end{proof}
%%%%%%%%%%%%%%%%%%
\section{The comparison of $\rmtlD^t(0,nx)$ and $\rmtlT_M^t(0,nx)$}\label{Sec: 4}
Given $t \geq 0$, let $\G^t([0]_{t},[nx]_{t})$ be the set of all geodesics of $\rmD^t([0]_{t},[nx]_{t})$. We recall that $\G^t_M([0]_{t},[nx]_{t})$ be the set of all geodesics of $\rmTmt([0]_{t},[nx]_{t})$. If $t=0$, we simply write  $\G([0],[nx])$ and  $\G_M([0],[nx])$ for compactness. In this section, we couple Bernoulli percolation with first truncated passage percolation in the sense that each edge is $t$-open (resp. $t$-closed) if and only if $\omega_e(t)=\tau_e(t)=1$ (resp. $\omega_e(t)= \infty$ and $\tau_e(t) =M$). The next proposition states that, with overwhelming probability, the geodesics of $\rmtlD^t(0,nx)$ and $\rmtlT_M^t(0,nx)$ coincide.
 \begin{proposition}\label{Prop: comparison of pitl and pikm}  
There exists a constant $C>0$ such that for all $t \geq 0$,
 \begin{align}
         \pr(\G^t([0]_{t},[nx]_{t})= \G^t_M([0]_{t},[nx]_{t})) \geq 1 - C\exp(-M^{3/4}/C).
    \end{align}
    As a consequence, we have
     \begin{align}\label{Eq: the difference of intersection}
         \E[| |\tlpi(0,nx) \cap \tlpit(0,nx)|- |\tlpim(0,nx) \cap \tlpimt(0,nx)||] \leq C\exp(-M^{3/4}/C).
    \end{align} 
\end{proposition}
%%%%%%%%%%%%%%%%%%%%
\begin{proof}[Proof of Proposition \ref{Prop: comparison of pitl and pikm}]
We first define two events
\begin{align*}
    & \kE_1^t:= \{ \|0-[0]_t\|_{\infty}, \|nx-[nx]_t\|_{\infty} \leq M^{7/8}\}; \\
     & \kE_2^t:=\{ \forall \gamma \in \G_M^t([0]_t,[nx]_t), \forall e \in \gamma: \re \leq M^{3/4} \}.
\end{align*}
Set $\kE_3^t: =\kE_1^t \cap \kE_2^t$. By Lemma \ref{Lem: hole}, there exists a constant $c>0$ such that
\begin{align*}
    \pr(\{\kE_1^t\}^c) \leq 2 \pr(\|0-[0]_t\|_{\infty}\geq M^{7/8}) \leq 2 c_1^{-1}\exp(-c_1M^{7/8}).
\end{align*}
Notice that on the event $\kE_1^t$, we have for all $\gamma \in \G_M^t([0]_t,[nx]_t)$, $|\gamma|\leq \rmtlTmt(0,nx)\leq 2n\|x\|_1M$, so that $\gamma \subset \Lambda_{4n\|x\|_1M}$. Therefore,
\begin{align*}
    \pr(\{\kE_2^t\}^c \cap \kE_1^t) & \leq \pr(\exists \gamma \in \G_M^t([0]_t,[nx]_t),\gamma \subset \Lambda_{4n\|x\|_1M}, \exists e \in \gamma: \re \geq M^{3/4}) \\
    & \leq \pr(\exists e \in  \Lambda_{4n\|x\|_1M}: \re \geq M^{3/4})\\
    & \leq c_2^{-1}(8n\|x\|_1M)^d \exp(-c_2M^{3/4}),
\end{align*}
for some constant $c_2>0$. Putting two above inequalities, with some constant $c_3 >0$,
\begin{align*}
    \pr(\{\kE_3^t\}^c) \leq c_3^{-1}\exp(-c_3M^{3/4}).
\end{align*}
Now we will show that on the event $\kE_3^t$,
\begin{align*}
   \kE_4^t:= \{\forall \gamma \in  \G_M^t([0]_t,[nx]_t): \text{ if } e \in \gamma \setminus \{\Lambda_{2M^{7/8}}(0) \cup \Lambda_{2M^{7/8}}(nx)\} \text{ then } e \text{ is $t$-open}\}.
\end{align*}
Indeed, we suppose the hypothesis that $\kE_3^t$ occurs, let $\gamma \in  \G_M^t([0]_t,[nx]_t)$, and $e \in \gamma \setminus \{\Lambda_{2M^{7/8}}(0) \cup \Lambda_{2M^{7/8}}(nx)\}$. It follows from $\kE_3^t$ that $\re \leq M^{3/4}$, so that $\Lambda_{3\re}(e) \cap \Lambda_{M^{7/8}}(0)= \varnothing$ and  $\Lambda_{3\re}(e) \cap \Lambda_{M^{7/8}}(nx) = \varnothing$ since $e \notin \Lambda_{2M^{7/8}}(0) \cup \Lambda_{2M^{7/8}}(nx)$. Thus, on the event $\kE_3^t$, we have $[0]_t \notin \Lambda_{3\re}(e)$ and $ [nx]_t \notin \Lambda_{3\re}(e)$ for $n$ large enough. Applying Proposition \ref{Propo: good properties of re} (iii) to $\gamma\in  \G_M^t([0]_t,[nx]_t)$, we obtain a modified path $\eta_e$ from $[0]_t$ to $[nx]_t$ such that 
\begin{align*}
  e \notin \eta_e,\quad  \eta_e \setminus \gamma \text{ is $t$-open}, \quad |\eta_e \setminus \gamma | \leq C_* \re.
\end{align*}
Hence,
\begin{align*}
    \rmT_M^t(\gamma) \leq   \rmT_M^t(\eta_e) =  \rmT_M^t(\eta_e \cap \gamma) + \rmT_M^t(\eta_e \setminus \gamma) \leq \rmT_M^t(\gamma) - \tau_e(t) + C_* \re,
\end{align*}
which implies that $\tau_e(t) \leq C_* \re < M$, so that $\tau_e(t) = 1$, i.e. $e$ is $t$-open.

Notice that by Lemmas \ref{Lem: hole} and \ref{lem: twodisjointclusters}, there exists a positive constant $C$ such that for all $z \in \Z^d$,
\begin{align*}
    \pr(\kE_5^t(z)) \geq 1 -C\exp(-M^{7/8}/C),
 \end{align*}
 where
\begin{align*}
  \kE_5^t(z):=  \{\kC_\infty^t  & \cap \Lambda_{M^{7/8}}(z) \neq \varnothing \} \cap   \{\forall \gamma_1, \gamma_2 \subset \mathbb{O}^t(\Lambda_{M^{7/8}}(z)): \Diam(\gamma_1),\Diam(\gamma_2)\geq M^{7/8},\gamma_1 \cap \gamma_2 \neq \varnothing\}.
\end{align*}
Moreover, on the event $\kE_4^t$, $\gamma$ has to cross the annuli $\rmA_{2M^{7/8}}(0)= \Lambda_{6M^{7/8}}(0)\setminus \Lambda_{2M^{7/8}}(0) $ and $\rmA_{2M^{7/8}}(nx):=\Lambda_{6M^{7/8}}(nx)\setminus \Lambda_{2M^{7/8}}(nx)$ by two $t$-open paths $\gamma_0$ and $\gamma_{nx}$, respectively. Therefore, on the event $\kE_4^t \cap \kE_5^t(0) \cap \kE_5^t(nx)$, these $t$-open paths have to intersect with $\kC_\infty^t$ at two points $u$ and $v$, respectively (see Figure \ref{de1}). Combining this with Lemma \ref{Lem: large deviation of graph distance Dlambda} yields that for some large constant $C>0$,
\begin{align*}
\pr(\kE_6^t) \geq 1-C\exp(-M^{3/4}/C),
\end{align*}
where
\begin{align*}
  \kE_6^t: = \{ \forall \gamma \in  \G_M^t([0]_t,[nx]_t), & \exists u \in \gamma \cap  \kC_\infty^t \cap \rmA_{2M^{7/8}}(0),\exists  v \in \gamma \cap \kC_\infty^t \cap \rmA_{2M^{7/8}}(nx), \\
  &\|u-[0]_t\|_\infty,  \|v-[nx]_t\|_\infty  \leq 6M^{7/8}: 
    \rmD^t([0]_t,u) \leq CM^{7/8},\rmD^t([nx]_t,v)\leq CM^{7/8} \}.
\end{align*}
\begin{figure}[htbp]
\begin{center}
\includegraphics[width=15cm]{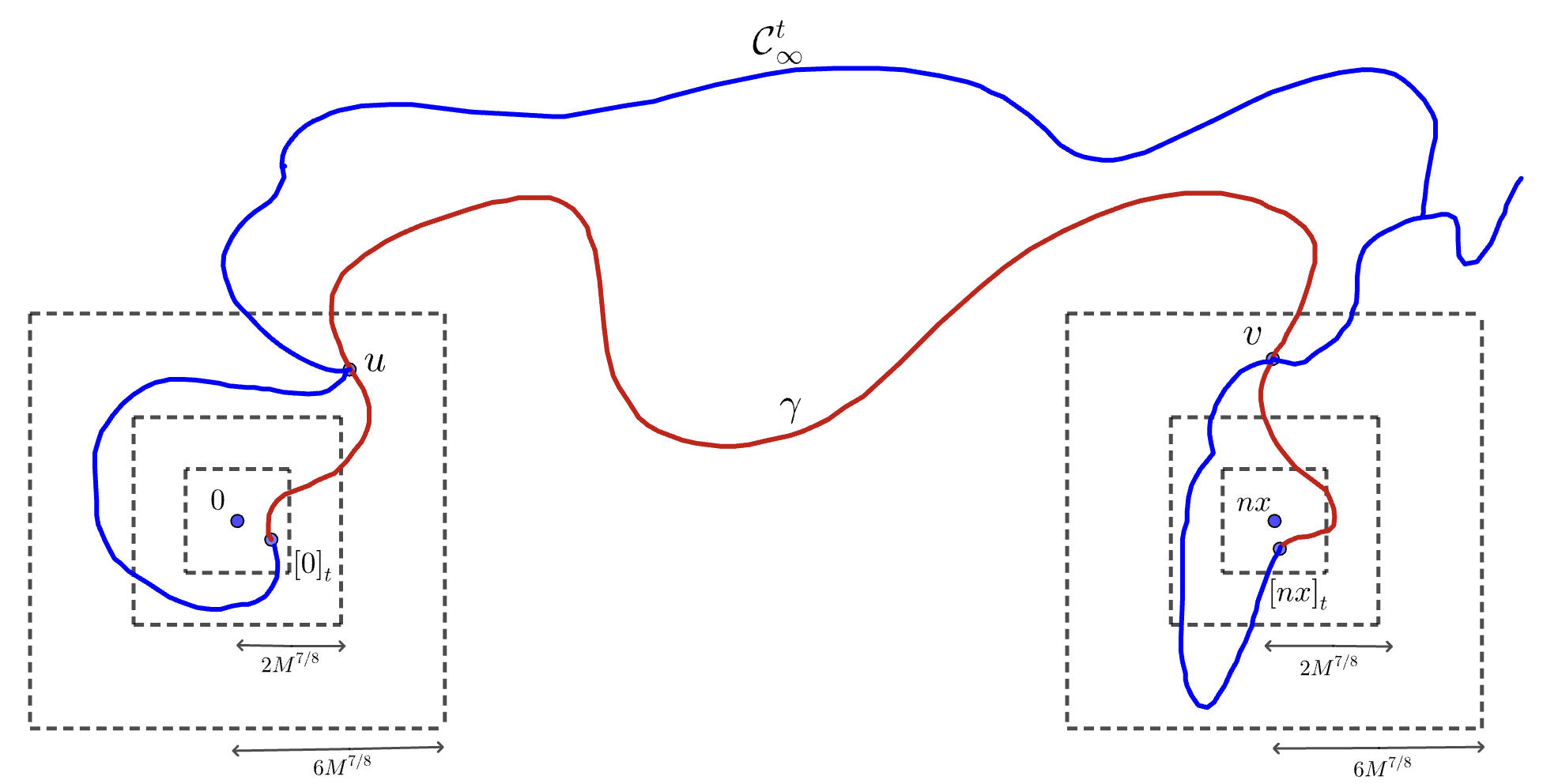}
\caption{Illustration of the event $\kE_6^t$} 
\label{de1}
\end{center}
\end{figure}
It now follows from the event $\kE_4^t \cap \kE_6^t$ that for each $\gamma \in \G_M^t([0]_t,[nx]_t)$, there exist $u \in \gamma \cap \rmA_{2M^{7/8}}(0), v \in \gamma \cap \rmA_{2M^{7/8}}(nx)$ such that sub-geodesics $\gamma_{[0]_t,u}$,$\gamma_{[nx]_t,v}$, and $\gamma \setminus \{\Lambda_{2M^{7/8}}(0) \cup \Lambda_{2M^{7/8}}(nx)\}$  are $t$-open (i.e. does not contain any $M$-weight edges), so that $\gamma$ is $t$-open and $\rmtlD^t(0,nx) \leq \rmD^t(\gamma)= \rmTmt(\gamma)< \infty$. On the other hand, it is clear that $\rmtlD^t(0,nx) \geq \rmtlTmt(0,nx)= \rmTmt(\gamma)$. Hence, on the event $\kE_4^t\cap \kE_6^t$, we have $\rmtlD^t(0,nx) = \rmD^t(\gamma)$, i.e. $\gamma$ is also a geodesic of $\rmtlD^t(0,nx)$, thus $ \G_M^t([0]_t,[nx]_t) \subset \G^t([0]_t,[nx]_t)$. Otherwise, take $\eta \in \G^t([0]_t,[nx]_t)$, so $\rmtlD^t(0,nx)= \rmD^t(\eta)= \rmTmt(\eta)$. Notice that by $\kE_4^t\cap \kE_6^t$, there exists a geodesic $\gamma \in \G_M^t([0]_t,[nx]_t) \subset \G^t([0]_t,[nx]_t)$ such that $\rmtlD^t(0,nx) = \rmD^t(\gamma)= \rmTmt(\gamma)= \rmtlTmt(0,nx)$. Therefore, we have $ \rmtlTmt(0,nx) = \rmTmt(\eta)$, i.e. $\eta$ has to be a geodesic of $\rmtlTmt(0,nx)$. This implies that $ \G^t([0]_t,[nx]_t) \subset \G_M^t([0]_t,[nx]_t)$. Hence, we get for all $t\geq 0$,
\begin{align*}
\pr(\G_M^t([0]_t,[nx]_t) = \G^t([0]_t,[nx]_t)) \geq \pr(\kE^t_4 \cap \kE^t_6) \geq 1-C\exp(-M^{3/4}/C),
\end{align*}
for some constant $C>0$.

Finally, we claim \eqref{Eq: the difference of intersection}. Using Lemma \ref{Lem: large deviation of graph distance Dlambda}, we obtain that
\begin{align*}
\max\{\E[ |\tlpit(0,nx)|^2],\E[ |\tlpimt(0,nx)|^2]\}= \cO(n^2).
\end{align*}
By Cauchy-Schwartz inequality,
\begin{align*}
  \E &[| |\tlpi(0,nx) \cap \tlpit(0,nx)|- |\tlpim(0,nx) \cap \tlpimt(0,nx)||] \\
  & \leq \E[| |\tlpi(0,nx) \cap \tlpit(0,nx)|- |\tlpim(0,nx) \cap \tlpimt(0,nx)||\\
  & \qquad \qquad \qquad \qquad \times \I(\{\G_M([0]_t,[nx]_t) = \G([0]_t,[nx]_t) \cap \G_M^t([0]_t,[nx]_t) = \G^t(0,nx)\}^c)] \\
  & \leq \cO(n)(2\pr(\G_M([0]_t,[nx]_t) \neq \G([0]_t,[nx]_t)))^{1/2} \leq \cO(n)\exp(-M^{3/4}/2C).
\end{align*}
Here for the last line  we have used 
\begin{align*}
\max\{\E[ |\tlpit(0,nx)|^2],\E[ |\tlpimt(0,nx)|^2]\}= \cO(n^2),
\end{align*}
followed by Lemma \ref{Lem: large deviation of graph distance Dlambda}. 
\end{proof}
%%%%%%%%%%%%%%%%%%%%%%%%%%%%
As a consequence, the variance and the total co-influence of $\rmtlD^t$ are well approximated by those of $\rmtlTmt$.
\begin{proposition}\label{Prop: camparison dtl and dkm}
There exists a constant $C>0$ such that for all $t \geq 0$,
    \begin{align}\label{Eq: var t-tm}
       |\Var(\rmtlD^t(0,nx)) - \var(\rmtlTmt(0,nx))| \leq C\exp(-M^{3/4}/C),
    \end{align}  
and 
\begin{align}\label{Eq: compare influence}
     \Bigg|\sum_{e \in \cE}\E\left[\Delta_e(\rmtlD,\rmtlD^t)\right] -\sum_{e \in \cE}\E\left[ \Delta_e(\rmtlTm,\rmtlT^t_M) \right]\Bigg| \leq C\exp(-M^{3/4}/C).
\end{align}
\end{proposition}
\begin{proof}[Proof of Proposition \ref{Prop: camparison dtl and dkm}]
We first prove \eqref{Eq: var t-tm}. By Cauchy-Schwartz inequality, 
\begin{align*}
   | \Var&(\rmtlD^t(0,nx))- \var(\rmtlTmt(0,nx))| \\
   & \leq 2\left(\E[(\rmtlD^t(0,nx)- \rmtlTmt(0,nx))^2](\Var(\rmtlD^t(0,nx))+ \var(\rmtlTmt(0,nx)))\right)^{1/2}.
\end{align*}
We remark that if $\G^t([0]_{t},[nx]_{t})= \G^t_M([0]_{t},[nx]_{t})$ then $\rmtlD^t(0,nx)= \rmtlTmt(0,nx)$. Therefore, there exists a constant $C>0$ such that for all $t \geq 0$,
\begin{align}\label{Eq: D= Tm}
     \pr(\rmtlD^t(0,nx)= \rmtlTmt(0,nx)) \geq \pr(\G^t([0]_{t},[nx]_{t})= \G^t_M([0]_{t},[nx]_{t})) \geq 1 - C\exp(-M^{3/4}/C).
\end{align}
 Thus by Proposition \ref{Prop: comparison of pitl and pikm} and Cauchy-Schwartz inequality, 
\begin{align}\label{Eq: diff d and t}
    \E[(\rmtlD^t(0,nx) &- \rmtlTmt(0,nx))^2]  =  \E[(\rmtlD^t(0,nx)- \rmtlTmt(0,nx))^2\I(\rmtlD^t(0,nx)\neq \rmtlTmt(0,nx))] \notag \\
    & \leq \left(8(\E[(\rmtlD^t(0,nx))^4] +\E[(\rmtlTmt(0,nx))^4]) \pr(\rmtlD^t(0,nx)\neq \rmtlTmt(0,nx))\right)^{1/2} \notag \\
    & \leq Cn^2\exp(-M^{3/4}/C),
\end{align}
for some constant $C>0$. Combining this with  $\max(\Var(\rmtlD^t(0,nx)),\var(\rmtlTmt(0,nx)))= \cO(n) $, we obtain \eqref{Eq: var t-tm}.
%%%%%%%%%%%%%%%%%%%%%%%%%%%%%%%
We now prove \eqref{Eq: compare influence}. For $e \in \cE$ and $t \geq 0$, we define
\begin{align*}
     & \cA_e:=\{\rmtlD(0,nx) = \rmtlT_M(0,nx), \rmtlD(0,nx) \circ \sigma^{\omega^1_e}_e(\omega)= \rmtlT_M(0,nx) \circ \sigma^{\tau^1_e}_e(\tau)\},
\end{align*}
and 
\begin{align*}
     & \cB_e^t:=\{\rmtlD^t(0,nx) = \rmtlT^t_M(0,nx), \rmtlD^t(0,nx) \circ \sigma^{\omega^2_e}_e(\omega(t))= \rmtlT^t_M(0,nx) \circ \sigma^{\tau^2_e}_e(\tau(t))\}.
\end{align*}
Using \eqref{Eq: D= Tm} and union bound, we have for some $C>0$,
\begin{align}\label{Eq: A0At}
    \pr(\{\cA_e \cap \cB_e^t\}^c) \leq  C\exp(-M^{3/4}/C).
\end{align}
Observe that on the event $\cA_e \cap \cB_e^t$, we have $\Delta_e(\rmtlD,\rmtlD^t) = \Delta_e(\rmtlT_M,\rmtlT^t_M)$. Thus,
\begin{align}\label{Eq: sum detae bound}
    \sum_{e \in \cE}\E\left[\Delta_e(\rmtlD,\rmtlD^t)\right] \leq \sum_{e \in \cE}\E\left[\Delta_e(\rmtlT_M,\rmtlT^t_M)\right] + \sum_{e \in \cE}\E\left[\Delta_e(\rmtlD,\rmtlD^t)\I(\{\cA_e \cap \cB^t_e\}^c)\right].
\end{align}
After straightforward algebra, we can write 
\begin{align} \label{Eq: p(1-p)}
    \E[|\Delta_e(\rmtlD,\rmtlD^t)| \I(\{\cA_e \cap \cB^t_e\}^c)] & =  \E[|\Inf_e (\rmtlD,\rmtlD^t)|\I(\omega_e^1 = \omega_e^2 = \infty, \omega_e = 1) \I(\{\cA_e \cap \cB^t_e\}^c)] \notag \\
    &  \qquad \qquad + \E[|\Inf_e (\rmtlD,\rmtlD^t)|\I(\omega_e^1 = \omega_e^2 = 1, \omega_e = \infty) \I(\{\cA_e \cap \cB^t_e\}^c)]  \notag \\
    & \leq  2\E[|\Inf_e (\rmtlD,\rmtlD^t)| \I(\{\cA_e \cap \cB^t_e\}^c)], 
\end{align}
where 
\begin{align*}
    \Inf_e(\rmtlD,\rmtlD^t) := \nabla_e^{\infty,1} \rmtlD(0,nx) \nabla^{\infty,1}_e \rmtlD^t(0,nx).
\end{align*}
Let $\omega'':=(\omega''_e)_{e \in \cE}$ be an independent copy of $\omega$ and independent of $\omega', \tau^1,\tau^2, \omega^1, \omega^2$. Observe that $\omega''_e$ is independent of two events $\{\cA_e \cap \cB^t_e\}$ and $\cQ_e^0$. Using the similar argument as in \eqref{Eq: delta1}, we obtain that
\begin{align*}
   0 \leq \Inf_e (\rmtlD,\rmtlD^t) \I(\omega''_e =1) \I(\cQ_e^0) & \leq  \Inf_e (\rmtlD,\rmtlD^t) \I(e \in \tlpi(0,nx))\I(\cQ_e^0) \leq  |\Inf_e (\rmtlD,\rmtlD^t)| \I(e \in \tlpi(0,nx)),
\end{align*}
where we remark that $\tlpi(0,nx)$ is the geodesics in environment $(\omega''_e, (\omega_{e'})_{e' \neq e})$. Therefore, we get
\begin{align}\label{Eq: sum delta_e-sumde}
    \sum_{e \in \cE}\E\left[|\Inf_e (\rmtlD,\rmtlD^t)| \I(\cQ_e^0)  \I(\{\cA_e \cap \cB^t_e\}^c) \right] 
    & \leq \frac{1}{p} \E\left[ \sum_{e \in \tlpi(0,nx)} |\Inf_e(\rmtlD,\rmtlD^t)| \I(\{\cA_e \cap \cB^t_e\}^c) \right].
\end{align}
 By Cauchy-Schwarz inequality,
 \begin{align}\label{Eq: sum of inf_e}
 \E & \left[ \sum_{e \in \tlpi(0,nx)}|\Inf_e(\rmtlD,\rmtlD^t)|\I(\{\cA_e \cap \cB^t_e\}^c)\right]  \leq \E\left[\sum_{e \in \tlpi(0,nx)} |\Inf_e (\rmtlD,\rmtlD^t)|\I(\{\cA_e \cap \cB^t_e\}^c)\I(|\tlpi(0,nx)|\leq C'n)\right] \notag \\
 & \qquad \qquad \qquad \qquad \qquad \qquad \qquad \qquad \qquad  +   \E\left[\sum_{e \in \tlpi(0,nx)}|\Inf_e(\rmtlD,\rmtlD^t)|\I(\{\cA_e \cap \cB^t_e\}^c)\I(|\tlpi(0,nx)| \geq C'n) \right] \notag \\
  & \leq \sum_{e \in \Lambda_{C'n}} \E\left[|\Inf_e (\rmtlD,\rmtlD^t)|\I(\{\cA_e \cap \cB^t_e\}^c)\right] + \sum_{\ell \geq C'n}\E\left[ \sum_{e \in \tlpi(0,nx)} |\Inf_e (\rmtlD,\rmtlD^t)|\I(\{\cA_e \cap \cB^t_e\}^c)\I(|\tlpi(0,nx)|=\ell) \right] \notag  \\
  & \leq\sum_{e \in \Lambda_{C'n}} \E\left[|\Inf_e (\rmtlD,\rmtlD^t)|\I(\{\cA_e \cap \cB^t_e\}^c)\right]+ \sum_{\ell \geq C'n}\sum_{e \in \Lambda_{d\ell}}\left(\E\left[\left(\Inf_e (\rmtlD,\rmtlD^t)\right)^2\I(\{\cA_e \cap \cB^t_e\}^c)\right] \pr(|\tlpi(0,nx)|=\ell) \right)^{1/2},
 \end{align}
 where $C'$ is the constant defined as in Lemma \ref{Lem: large deviation of graph distance Dlambda}.
 Thank to Cauchy-Schwarz inequality again, we have
 \begin{align*}
\E\left[\left(\Inf_e (\rmtlD,\rmtlD^t)\right)^2\I(\{\cA_e \cap \cB^t_e\}^c)\right] \leq \left(\E\left[\left(\Inf_e (\rmtlD,\rmtlD^t)\right)^4\right]\pr(\{\cA_e \cap \cB^t_e\}^c)\right)^{1/2}.
 \end{align*}
 It is clear that
 \begin{align*}
    \E\left[\left(\Inf_e (\rmtlD,\rmtlD^t)\right)^4\right]  =  \E\left[(\nabla_e^{\infty,1} \rmtlD(0,nx))^4 (\nabla^{\infty,1}_e \rmtlD^t(0,nx))^4\right] & \leq \left(\E\left[(\nabla_e^{\infty,1} \rmtlD(0,nx))^8\right]\E \left[(\nabla^{\infty,1}_e \rmtlD^t(0,nx))^8 \right]\right)^{1/2} \\
    & = \E\left[(\nabla_e^{\infty,1} \rmtlD(0,nx))^8 \right] =  \cO(n^8).
 \end{align*}
 Plugging this with \eqref{Eq: A0At} yields that
 \begin{align*}
     \E\left[\left(\Inf_e (\rmtlD,\rmtlD^t)\right)^2\I(\{\cA_e \cap \cB^t_e\}^c)\right] =  C\exp(-M^{3/4}/C),
 \end{align*}
 for some constant $C>0$. 
  Combining these inequalities with Lemma \ref{Lem: large deviation of graph distance Dlambda}, one has for some constant $C>0$ that
 \begin{align}
     \sum_{\ell \geq C'n}\sum_{e \in \Lambda_{d\ell}}\left(\E\left[\left(\Inf_e (\rmtlD,\rmtlD^t)\right)^2\I(\{\cA_e \cap \cB^t_e\}^c)\right]\right)^{1/2}\left(\pr(|\tlpi(0,nx)|=\ell) \right)^{1/2} 
     & \leq C\exp(-M^{3/4}/C).
 \end{align}
 Using two last inequalities, \eqref{Eq: sum delta_e-sumde}, and \eqref{Eq: sum of inf_e} give us
 \begin{align} \label{Eq: sum of co-influencexaebe}
      \sum_{e \in \cE}\E\left[|\Inf_e (\rmtlD,\rmtlD^t))|\I(\cQ_e^0) \I(\{\cA_e \cap \cB^t_e\}^c)\right]  \leq C\exp(-M^{3/4}/C),
 \end{align}
 for some constant $C>0$. Thanks to Cauchy-Schwarz inequality, we get for some constant $c>0$,
 \begin{align}\label{Eq: influence Q^C ABE}
      \sum_{e \in \cE}\E& \left[|\Inf_e (\rmtlD,\rmtlD^t))|\I(\{\cQ_e^0\}^c) \I(\{\cA_e \cap \cB^t_e\}^c)\right]  \leq \sum_{e \in \cE}\left(\E\left[|\Inf_e (\rmtlD,\rmtlD^t))|^2\I(\{\cQ_e^0\}^c) \pr(\{\cA_e \cap \cB^t_e\}^c)\right]\right) \notag \\
     &  \leq  \sum_{e \in \cE}\left(\E\left[|\Inf_e (\rmtlD,\rmtlD^t)|^4 \right]\right)^{1/2}\left(\pr(\{\cQ_e^0\}^c)\right)^{1/4} \left(\pr(\{\cA_e \cap \cB^t_e\}^c)\right)^{1/4}  \notag \\
     & \leq c^{-1}\sum_{e \in \cE}  \exp(-c\min(\|e\|_\infty,\|e-nx\|_{\infty}) \exp(-cM^{3/4}) \notag \\
     & \leq  c^{-1}\exp(-cM^{3/4}/2),
 \end{align}
where for the third inequality we have used \eqref{Eq: probability of Q^c} and Proposition \ref{Prop: comparison of pitl and pikm}. Combining two above upper bounds with \eqref{Eq: p(1-p)} and \eqref{Eq: sum detae bound} implies that there exists $C> 0$ such that
 \begin{align}\label{Eq: DDt<}
     \sum_{e \in \cE}\E\left[\Delta_e(\rmtlD,\rmtlD^t)\right] \leq \sum_{e \in \cE}\E\left[\Delta_e(\rmtlT_M,\rmtlT^t_M)\right] + C\exp(-M^{3/4}/C).
 \end{align}
 On the other hand,
 \begin{align*}
\sum_{e \in \cE}\E\left[\Delta_e(\rmtlT_M,\rmtlT^t_M)\right] \leq \sum_{e \in \cE}\E\left[\Delta_e(\rmtlD,\rmtlD^t)\right]  + \sum_{e \in \cE}\E\left[\Delta_e(\rmtlTm,\rmtlT^t_M)\I(\{\cA_e \cap \cB^t_e\}^c)\right].
\end{align*}
By a similar argument as in \eqref{Eq: sum of co-influencexaebe} and \eqref{Eq: influence Q^C ABE}, we deduce that for some constant $C>0$,
\begin{align*}
    \sum_{e \in \cE}\E\left[|\Delta_e(\rmtlTm,\rmtlT^t_M)|\I(\{\cA_e \cap \cB^t_e\}^c)\right] \leq C\exp(-M^{3/4}/C).
\end{align*}
Hence,
 \begin{align}\label{Eq: TTt<}
     \sum_{e \in \cE}\E\left[\Delta_e(\rmtlT_M,\rmtlT^t_M)\right] \leq \sum_{e \in \cE}\E\left[\Delta_e(\rmtlD,\rmtlD^t)\right]  + C\exp(-M^{3/4}/C).
 \end{align}
The result follows by merging \eqref{Eq: DDt<} and \eqref{Eq: TTt<}.
\end{proof}
%%%%%%%%%%%%%%%%%%%%%%%%%%%%%%%%%%%%%%%%%%
%%%%%%%%%%%%%%%%%%%%%%%%%%%%%%%%%%%%%%%%%%
\appendix
\section{Dynamical formula for variance: Proof of Lemma \ref{Lemma: covariacne formula}}\label{Appendix: DF for variance}
Given $h: \{\ell,L\}^{|E|}  \to \R$ and $e \in E$, let us consider the interpolating function between $X$ and $X'$, 
\begin{align*}
   \phi(t_e):= \E[h(X)h(X(t_e,(t_{e'})_{e'\neq e})].
\end{align*}
We first derive the relation between derivative of $\phi(t_e)$ and the co-influence at time $t_e$.
\begin{lemma}\label{Lemma: derivative of phi s}
 For any functions $h \in L^2(G_p^{|E|})$ and $e \in E$, we have $\phi(t_e)$ is non-increasing in $t_e$ and 
\begin{align}
   \phi'(t_e) =  -\E\left[\nabla^{X_e,X^1_e}_e h(X) \nabla^{X_e,X^2_e}_e h(X((t_e)_{e \in E}))\right]. 
\end{align}
\end{lemma}
\begin{proof}
    To compute $\phi'(t_e)$ with $t_e \in (0,1)$, we take $\delta > 0$ such that $t_e + \delta \in [0,1]$.
Write $Y : = X(t_e,(t_{e'})_{e' \neq e})$ and $Y':= X(t_e+ \delta, (t_{e'})_{e' \neq e})$. Then,
\begin{align}\label{Eq: diff phi+delta-phi}
    \phi(t_e+\delta)-\phi(t_e)=  \E[h(X) (h(Y')-h(Y))].
\end{align}
Notice that if $U_e \not \in [t_e,t_e +\delta]$ then $Y = Y'$. Otherwise, we have $Y = (X_e, (X_{e'}(t_{e'}))_{e'\neq e})$ and $Y' = (X^2_e, (X_{e'}(t_{e'}))_{e'\neq e})$ with remark that $X^2_e$ is the independent copy of $X_e$. Therefore,
\begin{align*}
    \E[h(X) (h(Y')-h(Y))] & = \E[h(X) (h(Y')-h(Y))\I(U_e \in [t_e,t_e + \delta])] \\
   & =- \delta \E\left[h(X) \nabla^{X_e,X^2_e}_e h(X((t_e)_{e \in E}))\right].
\end{align*}
By $X_e^1$ is independent of $X_e$ and $X_e^2$,
\begin{align*}
    \E\left[h \circ \sigma^{X^1_e}_e(X) \nabla^{X_e,X^2_e}_e h(X((t_e)_{e \in E}))\right]  & = \E \left[\E_{X_e,X^2_e} h \circ \sigma^{X^1_e}_e(X) \nabla^{X_e,X^2_e}_e h(X((t_e)_{e \in E}))\right] \\
    & =  \E \left[ h \circ \sigma^{X^1_e}_e(X) \E_{X_e,X^2_e}\left[\nabla^{X_e,X^2_e}_e h(X((t_e)_{e \in E}))\right]\right] = 0.
\end{align*}
Hence,
\begin{align}
 \E[h(X) (h(Y')-f(Y))] & = - \delta \E\left[(h(X)-h \circ \sigma^{X^1_e}_e(X)) \nabla^{X_e,X^2_e}_e h(X((t_e)_{e \in E}))\right] \notag \\
 & =- \delta   \E\left[\nabla^{X_e,X^1_e}_e h(X) \nabla^{X_e,X^2_e}_e h(X((t_e)_{e \in E}))\right].
\end{align}
Combining this with $\eqref{Eq: diff phi+delta-phi}$ give us
\begin{align*}
   \phi'(t_e) = \frac{d}{d t_e} \E[h(X) h(X(t_e,(t_{e'})_{e' \neq e}))]=  -\E\left[\nabla^{X_e,X^1_e}_e h(X) \nabla^{X_e,X^2_e}_e h(X((t_e)_{e \in E}))\right]. 
\end{align*}
Finally, by \cite[Lemma 11]{ahlberg2023chaos}, $\phi(t_e)$ is non-increasing, and thus $\phi'(t_e) \leq 0$. It yields that
\begin{align*}
    \E\left[\nabla^{X_e,X_e^1}_e h(X) \nabla^{X_e,X_e^2}_e h(X((t_e)_{e \in E}))\right] \geq 0.
\end{align*}
This completes the proof.
\end{proof}
\begin{proof}[Proof of Lemma \ref{Lemma: covariacne formula}]
We first observe that
\begin{align} \label{Eq: cov fx fxt}
    \Cov(f(X),f(X(t))) = \E[f(X) f(X(t))]-\E[f(X) f(X(1))] = \int_{t}^1- \frac{d }{ds} \E[f(X) f(X(s))]ds.
\end{align}
In the case that $E$ is finite, by the chain rule, we have
\begin{align}\label{Eq: efx fxt}
   \frac{d}{ds} \E[f(X) f(X(s))]  = \sum_{e \in E} \frac{d}{d t_e} \E[f(X) f(X((t_e)_{e \in \E}))] \Big|_{t_e = s,\, \forall e \in E}.
\end{align}
Applying Lemma \ref{Lemma: derivative of phi s} to $h = f$, $e \in E$, we obtain that
\begin{align*}
    \phi'(t_e) = \frac{d}{d t_e} \E\left[f(X) f(X((t_e)_{e \in \E}))\right]=  -\E\left[\nabla^{X_e,X^1_e}_e f(X) \nabla^{X_e,X^2_e}_e f(X((t_e)_{e \in E}))\right]. 
\end{align*}
Taking $t_e = s$ for all $e \in E$ and plugging into \eqref{Eq: efx fxt} and \eqref{Eq: diff phi+delta-phi}, we obtain \eqref{Eq of lem: covariance}. Since $$ \forall s \in [0,1], \quad \E\left[\nabla^{X_e,X^1_e}_e f(X) \nabla^{X_e,X^2_e}_e f(X(s))\right] \geq 0,$$ 
we have $\Cov(f(X),f(X(t)))$ is non-negative and non-increasing in $t \in [0,1]$.

If $E$ is (countable) infinite, we can enumerate its elements as $E=\{e_1,e_2,e_3,\ldots\}$. We define a sequence of $\sigma$-algebras by
\begin{align*}
   \mathcal{F}_0:= \{\varnothing\}; \quad \mathcal{F}_m:= \sigma(\{X_{e_1},X'_{e_1},U_{e_1}\},\ldots,\{X_{e_m},X'_{e_m},U_{e_m}\}).
\end{align*}
Now we consider a martingale
\begin{align*}
    g_m(X(t))= \E[f(X(t))|\cF_m].
\end{align*}
Applying the previous result for finite case, we obtain that 
\begin{align*}
      \Cov(g_m(X),g_m(X(t))) = \int_t^1 
\sum_{i=1}^m\E\left[\nabla^{X_{e_i},X_{e_i}^1}_{e_i} g_m(X) \nabla^{X_{e_i},X_{e_i}^2}_{e_i} g_m(X(s))\right] ds.
\end{align*}
By the assumption that $f$ has the finite second moment and the Doob martingale theorem, we have $g_m(X(t))$ convergence to $f$ in $L^2$. Therefore, we have
\begin{align*}
    \Cov(g_m(X),g_m(X(t))) \to \Cov(f(X),f(X(t))),
\end{align*}
as $m \to \infty$. Observe that the co-influence is always non-negative and for each $i \geq 1$ and $s \in [0,1]$, if $m  < i$, then $\E\left[\nabla^{X_{e_i},X_{e_i}^1}_{e_i} g_m(X) \nabla^{X_{e_i},X_{e_i}^2}_{e_i} g_m(X(s))\right] = 0$. Thus, to prove \eqref{Eq of lem: covariance}, thanks to the monotone convergence theorem, we remain to show that for each $i\geq 1$, $m \geq i$ and $s \in [0,1]$,
\begin{align}\label{Eq: influence convergence}
   \E\left[\nabla^{X_{e_i},X_{e_i}^1}_{e_i} g_m(X) \nabla^{X_{e_i},X_{e_i}^2}_{e_i} g_m(X(s))\right] \to \E\left[\nabla^{X_{e_i},X_{e_i}^1}_{e_i} f(X) \nabla^{X_{e_i},X_{e_i}^2}_{e_i} f(X(s))\right] \text{ as $m \to \infty$}. 
\end{align}
Indeed, for each $i \geq 1, $ $s \in [0,1]$, if $W$ and $Z$ are two copies of $X_{e_i}$ and independent of $(X_{e_j}(s))_{j \neq i}$ then by the Doob martingale theorem
\begin{align*}
 g_m \circ \sigma_{e_i}^{W}(X(s)) = \E\left[f \circ \sigma_{e_i}^{W}(X(s))|\cF_m \right] \to f \circ \sigma_{e_i}^{W}(X(s)); \, g_m \circ \sigma_{e_i}^{Z}(X(s)) = \E\left[f \circ \sigma_{e_i}^{Z}(X(s))|\cF_m \right] \to f \circ \sigma_{e_i}^{Z}(X(s)),
\end{align*}
in $L^2$ as $m \to \infty$. Therefore, 
\begin{align*}
    \E\left[g_m \circ \sigma_{e_i}^{Y}(X) g_m \circ \sigma_{e_i}^{Z}(X(s))\right] \to \E\left[f \circ \sigma_{e_i}^{Y}(X) f \circ \sigma_{e_i}^{Z}(X(s))\right], \text{ as } m \to \infty.
\end{align*}
This together with the definition of discrete derivative operator yields \eqref{Eq: influence convergence}.
\end{proof}
\section{Lower bound for the intersection of geodesics: Proof of Proposition \ref{Thm: intersection of geo}}
For each $z \in \Z^d$ and $N \geq 1$, we define a box 
    \begin{align*}
        B_N(z):= 2Nz+ [-N,N)^d \cap \Z^d.
    \end{align*}
It is clear that the family of boxes $(B_N(z))_{z \in \Z^d}$ builds a partition of $\Z^d$. Recall that $\omega'$ be an independent copy of $\omega$. For each $z\in \Z^d$, we define $\Bar{\omega}^{z}=(\Bar{\omega}^{z}_{e})_{e \in \cE}$ as
\begin{align*}
    \Bar{\omega}^z_e := \begin{cases}
        \omega_e', & \text{ if } e \in B_N(z); \\
        \omega_e, & \text{ if }   e \notin B_N(z).
    \end{cases}
\end{align*}
 In other words, $\Bar{\omega}^{z}$ is obtained by resampling $\omega$ in $B_N(z)$.
Notice that for any $z \in \Z^d$, $\omega$ and $\Bar{\omega}^z$ have the same distribution.

For convenience, in the remain part of proof, we write $\bO_\omega$, $\G_{\omega}$, $\rmD_\omega$ to emphasize the set of open paths, the set of geodesics, the graph distance in the environment $\omega$, respectively. Set $\kE_{n}:= \{ [0] \in \Lambda_{n^{1/2d}}(0), [nx] \in  \Lambda_{n^{1/2d}}(nx)\} $. For each $z \in \Z^d$, $\delta > 0$ and $u \neq v \in \partial B_N(z)$, we define
\begin{align*}
    \cA_z^{u,v}& := \{\text{there exists }\gamma \in \G_\omega([0],[nx]), \gamma \text{ crosses }B_N(z) \text{ at the first point $u$ and last point $v$}, \\
    & \qquad \qquad \qquad \qquad |\gamma_{u,v}| \geq (1+\delta)\|u-v\|_1 \};
\end{align*}
and
\begin{align*}
    \cB_z^{u,v}& := \{\text{there exists an open path }\pi_{u,v} \subset B_N(z) \cap \bO_{\Bar{\omega}^z}(B_N(z)) \text{ such that } |\pi_{u,v}|=\|u-v\|_1 
    \}.
\end{align*}
We note that the event $ \cB_z^{u,v}$ depends solely on the edge-weights inside $B_N(z)$ or $(\Bar{\omega}^z_e)_{e \in B_N(e)}$. Thus, $ \cB_z^{u,v}$ is independent of $\cA_z^{u,v}$ and $\kE_{n}$. 
We will show that for all $z \not \in \Lambda_{0,nx}:= \Lambda_{n^{1/2d}+2N}(0) \cup  \Lambda_{n^{1/2d}+2N}(nx) $,
\begin{align} \label{Eq: a and b}
    \kE_{n} \cap \cA_z^{u,v} \cap \cB_z^{u,v} \subset \{\exists \eta' \in \kP(B_N(z)): \forall \eta \in \G_{\Bar{\omega}^z} ([0],[nx]),\, \eta' \subset \eta\}.
\end{align}
Indeed, we suppose now that for $z \not \in \Lambda_{0,nx}$, $\kE_{n} \cap \cA_z^{u,v} \cap \cB_z^{u,v}$ occurs.  
Observe that the regularized points $[0],[nx]$ coincide in the environments $\omega$ and $\Bar{\omega}^z$ since $[0] \in \Lambda_{n^{1/2d}}(0), [nx] \in  \Lambda_{n^{1/2d}}(nx)$ and $ \{\Lambda_{n^{1/2d}}(0) \cup  \Lambda_{n^{1/2d}}(nx)\}\cap B_{N}(z) = \varnothing$.
Then, there exists $\gamma \in \G_\omega([0],[nx])$ crossing $B_N(z)$ at the first point $u$ and the last point $v$, and an open path $\pi_{u,v} \subset B_N(z) \cap \bO_{\Bar{\omega}^z}(B_N(z))$
such that $|\gamma_{u,v}| > \|u-v\|_1 = |\pi_{u,v}|$. Therefore,
\begin{align*}
\rmD_{\omega}(\gamma) = \rmtlD_{\omega}(0,nx)  & = |\gamma_{[0],u}|+|\gamma_{v,[nx]}|+|\gamma_{u,v}| 
 > |\gamma_{[0],u}|+|\gamma_{v,[nx]}|+|\pi_{u,v}|  \geq \rmtlD_{\Bar{\omega}^z}(0,nx).
\end{align*}
We remark  that for all paths $\eta$ such that $\eta \cap B_N(z) = \varnothing$,
$\rmD_{\omega}(\eta)= \rmD_{\Bar{\omega}^z}(\eta),$
and thus
\begin{align*}
   \rmD_{\Bar{\omega}^z}(\eta) =  \rmD_\omega(\eta) \geq  \rmtlD_\omega(0,nx) >  \rmtlD_{\Bar{\omega}^z}(0,nx).
\end{align*}
Hence, on the event $\kE_n \cap \cA_z^{u,v}\cap \cB_z^{u,v}$, one has for all paths $\eta \in \G_{\Bar{\omega}^z}([0],[nx])$, there exits $\eta' \in  \kP(B_N(z))$ such that $\eta' \subset \eta$ and so we claim \eqref{Eq: a and b}.

Using  \eqref{Eq: a and b} and the fact that $\omega$ and $\Bar{\omega}^z$ has the same distribution, we have
\begin{align}\label{Eq: AuvBuv}
    \sum_{z \not \in  \Lambda_{0,nx}} \sum_{u \neq v \in \partial B_{N}(z)} \pr( \kE_{n} \cap \cA_z^{u,v} & \cap  \cB_z^{u,v})  \leq   \sum_{z \in \Z^d} \sum_{u,v \in \partial B_{N}(z)} \pr(\exists \eta' \in \kP(B_N(z)): \forall \eta \in \G_{\Bar{\omega}^z}([0],[nx]), \eta' \subset \eta) \notag \\
    & = |\partial B_{N}(z)|^2 \sum_{z \in \Z^d}  \pr(\exists \gamma' \in \kP( B_N(z)): \forall \gamma \in \G_{\omega}([0],[nx]), \gamma' \subset \gamma) \notag \\
    & \leq |\partial B_{N}(z)|^2 \sum_{z \in \Z^d}  \pr(\exists e \in  B_N(z): \forall \gamma \in \G_{\omega}([0],[nx]), e \subset \gamma) \notag \\
     & \leq (2N+1)^{2d} \sum_{z \in \Z^d}  \pr(\exists e \in  B_N(z): \forall \gamma \in \G_{\omega}([0],[nx]), e \subset \gamma) \notag \\
     & \leq (2N+1)^{2d} \sum_{z \in \Z^d}  \sum_{e \in B_N(z)} \pr(\forall \gamma \in \G_{\omega}([0],[nx]), e \subset \gamma) \notag \\
     & = (2N+1)^{2d} \sum_{e \in \cE} \pr(\forall \gamma \in \G_{\omega}([0],[nx]), e \subset \gamma) \notag \\
     & = (2N+1)^{2d} \E[|\tlpi(0,nx)|].
\end{align}
We remain lower bound for the left hand side of \eqref{Eq: AuvBuv}. For each $z \in \Z^d$ and $\rho, \delta>0$, we say that a box $B_N(z)$ is $(\rho,\delta)$-\textit{good} if for all $u \neq v \in B_N(z)$ with $\rmD_\omega(u,v) < \infty$, then $\rho\|u-v\|_1 \geq \rmD(u,v) \geq (1+\delta)\|u-v\|_1$.  Notice that
$\pr(\cB_z^{u,v}) \geq p^{\|u-v\|_1} \geq p^{d(2N+1)}$ for all $u \neq v \in \partial B_N(z)$. By independent property, we have for all $n$ large enough
\begin{align}\label{lowbound auvzbuvz}
     \sum_{z \not \in  \Lambda_{0,nx}}&    \sum_{u \neq v \in \partial B_{N}(z)} \pr( \kE_{n} \cap \cA_z^{u,v} \cap  \cB_z^{u,v})  \geq   \sum_{z \not \in \Lambda_{0,nx}} \sum_{u \neq v \in \partial B_{N}(z)} \pr( \kE_{n} \cap \cA_z^{u,v})   \pr(\cB_z^{u,v}) \notag \\
     & \geq  p^{d(2N+1)} \sum_{z \not \in \Lambda_{0,nx}} \sum_{u \neq v \in \partial B_{N}(z)} \pr( \kE_{n}  \cap \{\exists \gamma \in \G_\omega([0],[nx]), \gamma \text{ crosses }B_N(z) \text{ at first point $u$ and last point $v$},\notag \\
     & \qquad \qquad \qquad \qquad \qquad \qquad \qquad \qquad \qquad \qquad |\gamma_{u,v}| \geq (1+\delta)\|u-v\|_1\})  \notag \\
     & \geq p^{d(2N+1)} \sum_{z \not \in \Lambda_{0,nx}} \pr(\kE_{n} \cap \{ \exists u \neq v \in \partial B_N(z), \exists \gamma \in \G_\omega([0],[nx]), \gamma \text{ crosses }B_N(z) \text{ at first point $u$ } \notag \\
     & \qquad \qquad \qquad \qquad \qquad \qquad \qquad \text{ and last point $v$, } \rmD_\omega (u,v) \geq (1+\delta)\|u- v\|_1\}) \notag \\
     & \geq \frac{1}{2}p^{d(2N+1)}\sum_{z \in \Z^d} \pr(\{\exists \gamma \in \G_\omega([0],[nx]), \gamma \text{ crosses }B_N(z), B_N(z) \text{ is $(\rho,\delta)$-good}\}|\kE_{n}) - \cO(n^{1/2}).
\end{align}
Here for the last line we have used that $\pr(\kE_{n}) \geq 1-C\exp(-n^{1/2d}/C)$ for some constant $C>0$. Given $[0] \in \Lambda_{n^{1/2d}}(0)$ and  $[nx] \in \Lambda_{n^{1/2d}}(nx)$, let $\eta$ be the geodesic from $[0]$ to $[nx]$ in environment $\omega$ with a deterministic rule breaking ties. The sum of the right hand side of $\eqref{lowbound auvzbuvz}$ is lower bounded by
\begin{align}\label{Eq: E of cross boxes}
    \sum_{z \in \Z^d} \pr(\eta  \text{ crosses }B_N(z), B_N(z) \text{ is $(\rho,\delta)$-good}|\kE_{n}) =  \E \left[\# \{\text{distinct $(\rho,\delta)$-good boxes that $\eta$ crosses}\}|\kE_{n} \right].
\end{align}

By \cite[Theorem 1.1]{jacquet2024strict}, Lemma \ref{Lem: large deviation of graph distance Dlambda} and union bound, there exist some constants $C,\rho, \delta>0$ such that for all $z \in \Z^d$ and $N$ large enough,
\begin{align}
    \pr(B_N(z) \text{ is $(\rho,\delta)$-good}) \geq 1 - C\exp(-N/C).
\end{align}
Notice that the event $\{B_N(z) \text{ is $(\rho,\delta)$-good}\}$ depends only on states of edges in $\Lambda_{\rho N}(z)$. Therefore, the variables $ \{\I(B_N(z) \text{ is $(\rho,\delta)$-good})\}_{z \in \Z^d}$ form a finitely dependent site percolation process. Using a standard Peierls argument and union bound (see, \cite[Lemma 5.2]{van1993inequalities}) to deduce that there exists $C>0$ such that  for all $N$ large enough
\begin{align*}
    \pr\left(\{\exists \text{ a path from }[0] \text{ to } [nx] \text{ that crosses at most } \frac{n}{CN} \text{ distinct $(\rho,\delta)$-good boxes}\}|\kE_n\right) \leq C\exp\left(\frac{-n}{CN}\right).
\end{align*}
It yields that  there exists a constant $c>0$ such that for all $N$ large enough,
\begin{align*}
     \E \left[\# \{\text{distinct $(\rho,\delta)$-good boxes that $\eta$ crosses}\}|\kE_n \right] \geq cn/N.
\end{align*}
Combining this with \eqref{Eq: E of cross boxes} and \eqref{lowbound auvzbuvz} give us 
\begin{align}
     \sum_{z \not \in \Lambda_{0,nx}} &  \sum_{u,v \in \partial B_{N}(z)} \pr( \kE_{n} \cap \cA_z^{u,v} \cap  \cB_z^{u,v})  \geq n/C,
\end{align}
for some constant $C>0$. Finally, merging this with \eqref{Eq: AuvBuv}, we have the desired result. 
\qed
%%%%%%%%%%%%%%%%%%%
\medskip
\bibliographystyle{alpha}
\bibliography{citation} 
\end{document}